\documentclass{amsart}

\RequirePackage[colorlinks,citecolor=blue,urlcolor=blue,linkcolor=blue]{hyperref}

\usepackage{amsmath}
\usepackage{amssymb}

\usepackage{enumitem}

\makeatletter
\@namedef{subjclassname@2020}{%
	\textup{2020} Mathematics Subject Classification}
\makeatother

\numberwithin{equation}{section}
\newtheorem{theorem}{Theorem}[section]
\newtheorem{proposition}[theorem]{Proposition}
\newtheorem{remark}[theorem]{Remark}
\newtheorem{definition}[theorem]{Definition}

\newtheorem{example}[theorem]{Example}
\newtheorem{lemma}[theorem]{Lemma}

\newtheorem{case}{Case}

\newtheorem{claim}{Claim}

\def\var{\mathbb{V}\mathrm{ar}}

\newcommand{\la}{\lambda}

\newcommand{\RR}{\mathbb{R}}

\newcommand{\bg}{\mathbf{g}}
\newcommand{\bh}{\mathbf{h}}
\newcommand{\bH}{\mathbf{h}}

\newcommand{\bw}{\mathbf{w}}
\newcommand{\ua}{\underline{a}}
\newcommand{\ub}{\underline{b}}
\newcommand{\R}{\mathbb{R}}

\newcommand{\cF}{\mathcal F}
\newcommand{\E}{\mathbb E}
\newcommand{\oL}{\ltimes}
\newcommand{\oR}{\rtimes}

\begin{document}
	
	\baselineskip=17pt
	
	
	\title[Flows in near algebras]{Flows in near algebras \\ with applications to harnesses}

	\author{W\l odzimierz Bryc}
	\address{
		Department of Mathematical Sciences,
		University of Cincinnati,
		Cincinnati, OH 45221--0025, USA}
	\email{Wlodzimierz.Bryc@UC.edu}
	
	\author{Jacek Weso\l owski}
	\address{ Faculty of Mathematics and Information Science
		Warsaw University of Technology,
		Warszawa, Poland}
	\email{wesolo@mini.pw.edu.pl}

	\author{Agnieszka Zi\c{e}ba}
	\address{ Faculty of Mathematics and Information Science
		Warsaw University of Technology, Warszawa, Poland}
	\email{agnieszka.zieba@mini.pw.edu.pl}
	
	\date{\today}
	
	\begin{abstract}
		We introduce   one-way flows in  near algebras and two-way flows in  double near algebras with two interrelated
		multiplications. We establish parametric representations of the one-way
		and two-way flows    in terms of  a single
		element of the algebra  that we call a flow generator. We  indicate  probabilistic applications of  the  one-way flows to a
		study of polynomial stochastic processes.  We apply  our results on  the  two-way flows to harnesses and quadratic
		harnesses in probability theory,
		generalizing some previous results. %
	\end{abstract}

	

	\maketitle
	
	\section{Introduction}
	In this paper we  study near algebras and related
	algebraic structures motivated by the probabilistic concepts of  polynomial processes,  harnesses and quadratic harnesses  with martingales being their common prefiguration.
	Martingales  are  fundamental in stochastic analysis, so we first explain how  algebraic techniques arise in the study of martingale-like stochastic processes.
	
	Recall that an integrable stochastic process $(X_t)_{t\ge 0}$ defined on a probability space $(\Omega,\cF,\mathbb P)$ and adapted to a filtration $(\cF_t)_{t\ge 0}$ is a martingale if $\E(X_t|\cF_s)=X_s$ for all $0\le s\le t$. Consider a potentially wider family of stochastic processes $(X_t)_{t\ge 0}$ satisfying the condition $\E(X_t|\cF_s)=\alpha_{st}X_s+\beta_{st}$, for some non-random coefficients $\alpha_{st},\;\beta_{st}$, $0\le s\le t$. Are martingales the only examples of this family?  If other  choices of coefficients besides  $(\alpha_{st},\,\beta_{st})\equiv (1,0)$ (the martingale case) are  possible, can we describe somehow their generic form? These questions, due to the tower property of conditional expectation, can be reformulated through the system of equations
	\begin{equation}\label{toy}
		(\alpha_{su},\,\beta_{su})=(\alpha_{tu}\alpha_{st},\,\alpha_{tu}\beta_{st}+\beta_{tu})=:(\alpha_{st},\beta_{st})\boxdot(\alpha_{tu},\beta_{tu}),
	\end{equation}
	holding for $0\le s< t< u$, where $\boxdot$ denotes the  binary operation in $\R^2$ defined through the middle term in \eqref{toy}. Actually, $(\R^2,+,\boxdot)$ is a simple example of an algebraic structure known as near algebra.
	This is a toy example of a flow in the abstract near algebras, which we analyze, as a warm-up,  in Section \ref{S-warmup}.

	It appears that near algebras provide a  specially  useful base for studying  properties of a  general class of stochastic processes with polynomial conditional moments  when conditioning is with respect to the past-future filtration $(\cF_{s,u})_{0\le s<u}$, as it is for   harnesses and quadratic harnesses  that  we concentrate on in this paper.
	The  algebraic structure which is convenient for identification of the  polynomial  conditional moments  is a double near algebra (i.e.,  two near algebras on the same linear space with suitably related multiplications).
	
	In the simplest case, i.e.,  when $\E(X_t|\cF_{s,u})=\alpha_{tsu}X_s+\beta_{tsu}X_u$, $0\le s<t<u$,  the analog of \eqref{toy} is
	\begin{align*}
		\label{toy1}
		&(\alpha_{tru},\beta_{tru})=(\alpha_{tsu}\alpha_{sru},\,\alpha_{tsu}\beta_{sru}+\beta_{tsu})=:(\alpha_{sru},\,\beta_{sru})\boxdot_1(\alpha_{tsu},\,\beta_{tsu}),\\
		&(\alpha_{sru},\,\beta_{sru})=(\alpha_{srt}+\alpha_{tru}\beta_{srt},\,\beta_{srt}\beta_{tru})=:(\alpha_{tru},\,\beta_{tru})\boxdot_2(\alpha_{srt},\,\beta_{srt}),
	\end{align*}
	holding for $0\le r<s<t<u$, where $\boxdot_1$ and $\boxdot_2$ are  binary operations on $\R^2$.  Actually, $(\R^2,+,\boxdot_i)$, $i=1,2$,
	are two near algebras on the same real linear space $(\R^2,+)$  with interrelated multiplications. This is a toy example of  a (two-way)  flow  in the  abstract double near algebras, which we introduce and analyze  in  Section \ref{EDNA}.
	
	As we have already mentioned, near algebra is the starting point of our considerations, so let us recall the definition of this algebraic structure: %

	\begin{definition}
		Let $V$ be a set with two binary operations $+$ and $\boxdot$. We say that $(V,+,\boxdot)$ is a {\em near algebra} if
		\begin{enumerate}
			\item $(V,+)$ is a linear space over $\mathbb{R}$,
			\item multiplication $\boxdot$ is associative, i.e.
			\begin{equation}
				({\bf x}\boxdot {\bf y}) \boxdot {\bf z}={\bf x}\boxdot({\bf y}\boxdot {\bf z}) \textnormal{ for all }{\bf x},
				{\bf y},{\bf z} \in V,
				\label{lacznosc}
			\end{equation}
			\item multiplication $\boxdot$ is left-distributive   with respect to
			addition, i.e.
			\begin{equation}
				{\bf x}\boxdot({\bf y} +{\bf z})={\bf x}\boxdot {\bf y}+{\bf x} \boxdot {\bf z} \textnormal{ for all }{\bf x},
				{\bf y}, {\bf z} \in V,
				\label{ml}
			\end{equation}
			\item multiplication $\boxdot$ is left-homogeneous, i.e.,
			\begin{equation}
				{\bf x}\boxdot (\la {\bf y})=\la ({\bf x}\boxdot {\bf y}) \textnormal{ for all }{\bf x}, {\bf y} \in V
				\textnormal { for all } \la \in \mathbb{R}.
				\label{L-jednorodn}
			\end{equation}
		\end{enumerate}
	\end{definition}
	We consider only near algebras with multiplicative identity, i.e., with a $\boxdot$-neutral element  $\mathbf{e}_\boxdot
	\in V$ which satisfies ${\bf x}\boxdot \mathbf{e}_\boxdot=\mathbf{e}_\boxdot \boxdot {\bf x}={\bf x}$ for all ${\bf x}\in
	V$. We denote by ${\bf x}^{-\boxdot}\in V$ the $\boxdot$-inverse of ${\bf x}\in V$, if it exists,
	as
	the unique element in V satisfying ${\bf x}\boxdot {\bf x}^{-\boxdot}={\bf x}^{-\boxdot}\boxdot {\bf
		x}=\mathbf{e}_\boxdot$.

	Near algebras were introduced independently  by
	Yamamuro in \cite{yamamuro1965near}  to study mappings on Banach
	spaces, and  later by
	Brown   \cite{bib_BrownNearAlgebras}. They fall into general hierarchy of mappings on algebraic structures including near-rings and near-fields, see Sect. 1 of Pilz \cite{Pilz96}.
	For a recent exposition and additional references see \cite{srinivas2017near}.
	
	Our goal in this paper is to use near algebras to study mappings that arise in the theory of stochastic processes which  following
	Hammersley \cite{Hammersley:1967aa} we call  harnesses. In the simplest case harnesses are integrable processes adapted to the  past-future filtration $(\cF_{su})_{0\le s<u}$
	satisfying
	\begin{equation}\label{LR0}
		\E(\tfrac{X_u-X_t}{u-t}|\cF_{s,u})=\tfrac{X_u-X_s}{u-s},\quad 0\le s<t<u.
	\end{equation}
	Williams \cite{Williams:1973aa} analyzed
	harnesses with finite second moments and finite number of non-zero coefficients and related them to random walks;
	harnesses that allow  infinite number of coefficients were studied in \cite{matysiak2008generalized}.  Kingman
	\cite{Kingman:1985aa,Kingman:1986aa} studied somewhat paradoxical properties of harnesses in the absence of second
	moments.
	In an unpublished note Williams characterized Wiener process as the harness with continuous trajectories; several authors
	\cite{Dozzi:1981aa,Dozzi:1989aa,Zhou:1992aa,Zhuang:1988aa} extended  Williams' result to multi-parameter setting.
	
	Quadratic harnesses are square-integrable harnesses satisfying additionally
	\begin{equation}\label{QH0}
		\E(X_t^2|\cF_{s,u})=Q_{tsu}(X_s,X_u),
	\end{equation}
	where $Q_{tsu}(x,y)$ is a polynomial of degree two in variables $x$ and $y$, see, e.g., \cite{bib_BrycMatysiakWesolowski},  \cite{Bryc-Wesolowski-10} and \cite{PJS14}. %
	
	We emphasize that the concrete form of the coefficients in polynomial conditional expectations  determines  properties of the process.  For
	example, quadratic  harnesses are often uniquely determined by
	conditions  \eqref{LR0} and \eqref{QH0}.
	Moreover, harnesses and quadratic harnesses include  important  families of stochastic processes such as the Wiener, Gamma,  and the Poisson processes.  Other examples of quadratic harnesses are related  to non-commutative probability
	\cite{bib_BrycWesolowski} and to Askey-Wilson polynomials \cite{Bryc-Wesolowski-08}.

	The paper is organized as follows.  In Section \ref{S-warmup} we  analyze
	the one-way flows in near algebras.
	In  Section \ref{EDNA} we introduce double near algebras and the two-way
	flows.
	The main results of the
	paper, Theorem \ref{Tw o  h} and its converse Theorem \ref{T-flow-exists},  describe
	the two-way flow in terms of a single element  of the double near algebra   (called a flow generator).
	In Section \ref{Sect:AP} we apply abstract
	algebraic results of Sections \ref{S-warmup} and \ref{EDNA} to stochastic processes with polynomial
	conditional moments.

	\section{Flows in  near algebras - a warm-up}\label{S-warmup}
	In this section,  we shall write $\mathbf{e}=\mathbf{e}_\boxdot$ for the identity and $\bf x^{-1}=\bf x^{-\boxdot}$ for the inverse of $\bf x$.
	\begin{definition}
		We will say that $\mathbf{f}\in V$ is a {\em $\boxdot$-null element}  if  for all $\mathbf{x}\in V$ we have
		\begin{equation}
			\label{A6}
			{\bf x}\boxdot {\bf f}={\bf f}.
		\end{equation}
	\end{definition}
	Referring to $\boxdot$-null elements we will also omit the symbol of multiplication in the notation in this section.
	
	Note that  zero of $(V,+)$ is a null element: by \eqref{L-jednorodn} for all $\mathbf{x}\in V$ $${\bf x}\boxdot
	{\bf 0}={\bf x}\boxdot(0\cdot{\bf 0})=0({\bf x}\boxdot {\bf 0})={\bf 0},$$ though for ${\bf x}\neq {\bf 0}\in V$ it may
	happen (see e.g. Section 2.1) that  ${\bf 0}\boxdot {\bf x}\neq{\bf 0}.$  It is easy to see that a null element cannot be
	invertible.  In Section \ref{Ex_of_DNA} we
	consider a near algebra where all non-invertible elements are
	$\boxdot$-null.
	
	The following formulas  will be used here and in Section \ref{EDNA}.
	\begin{proposition} \label{Prop-P1} Suppose $\bf x$ is invertible, $\bf f$ is null and $\alpha,\beta,\gamma\in\RR$.
		\begin{enumerate}
			\item Then  $\bf x+\bf f$ is invertible and
			\begin{equation}\label{xf}(\bf x+ \bf f)^{-1}=(\mathbf{e}-\bf f)\boxdot \mathbf{x}^{-1}.\end{equation}
			\item  If  $\alpha \mathbf{x}+\beta \mathbf{e}$ or $\beta \mathbf{x}^{-1} +\alpha \mathbf{e}$ is invertible,   then
			$\alpha \mathbf{x}+\beta \mathbf{e}+\gamma \mathbf{f}$ is invertible and
			
			\begin{equation}
				\label{Jacek}
				\beta( \alpha \mathbf{x}+\beta \mathbf{e}+\gamma \mathbf{f})^{-1}+\alpha (\beta \mathbf{x}^{-1} +\alpha
				\mathbf{e}+\gamma\mathbf{f})^{-1}+\gamma\mathbf{f}=\mathbf{e}.
			\end{equation}
		\end{enumerate}
	\end{proposition}
	\begin{proof}  Since the formula \eqref{xf} is straightforward, we will prove only  \eqref{Jacek}.
		Since $\gamma\mathbf{f}$ is a null element, see \eqref{L-jednorodn},
		it follows by (i)
		that at least one of
		$\alpha \mathbf{x}+\beta \mathbf{e}+\gamma \mathbf{f}$ and $\beta {\bf x}^{-1}+\alpha \mathbf{e} +\gamma \mathbf {f}$ is
		invertible. Since $\beta {\bf x}^{-1}+\alpha \mathbf{e} +\gamma \mathbf {f} =\mathbf{x}^{-1}\boxdot(\beta
		\mathbf{e}+\alpha\mathbf{x}+\gamma\mathbf{f})$ it follows that both are invertible.
		Using \eqref{L-jednorodn} we get
		$$\alpha (\beta {\bf x}^{-1}+\alpha \mathbf{e} +\gamma \mathbf {f} )^{-1}=(\alpha\,{\bf x}+\beta\,\mathbf{e} +\gamma\,{\bf
			f})^{-1}\boxdot(\alpha \mathbf{x}).$$
		By \eqref{A6} we also have
		$$
		\gamma\,{\bf f}=(\alpha\,{\bf x}+\beta\,\mathbf{e} +\gamma\,{\bf f})^{-1}\boxdot(\gamma\,{\bf f}).
		$$
		Plugging these two expressions  into the left hand side of \eqref{Jacek}  and using left distributivity of $\boxdot$ we
		rewrite the left hand side of \eqref{Jacek} as
		$$
		(\alpha {\bf x}+\beta \mathbf{e} +\gamma {\bf f})^{-1}\boxdot(\alpha {\bf x}+\beta \mathbf{e}+\gamma {\bf f}),
		$$
		and  the proof is complete.
	\end{proof}

	A family  $({\bf x}_{st})_{0\le s<t}$\hspace{1mm} of invertible elements of a near algebra $(V,+,\boxdot)$ that satisfies
	the  system of equations
	\begin{equation}\label{proflow}
		{\bf x}_{st}\boxdot {\bf x}_{tu}={\bf x}_{su},\quad 0\le s<t<u
	\end{equation}
	will be called a {\em one-way} flow on $[0,\infty)$ in $V$, if
	for all $ 0\le s<u$
	\begin{equation}\label{prolin}
		\tfrac{{\bf x}_{su}-\mathbf{e}}{u-s} \quad \mbox{does not depend on $u$.}
	\end{equation}

	Our goal is to show that   such  a family  is determined uniquely by a single element $\bh \in V$ which one could call  a
	flow generator.
	
	\begin{theorem}\label{protI}
		Suppose that  $({\bf x}_{st})_{0\le s<t}$\hspace{1mm} is a one-way flow in $V$.
		Then  there exists  unique   $\bh\in V$ such that $\mathbf{e}+s\bh$ is $\boxdot$-invertible for all  $s\ge 0$ and %
		\begin{equation}\label{qts}
			{\bf x}_{st}=(\mathbf{e}+s\bh)^{-1}\boxdot(\mathbf{e}+t\bh)%
			\qquad \mbox{ for all $0\le s<t$}.
		\end{equation}
		Conversely, if  $\bh\in V$ is such that $\mathbf{e}+s\bh$ is $\boxdot$-invertible for all  $s\ge 0$ then  $({\bf
			x}_{st})_{0\le s<t}$   given  by \eqref{qts} is a one-way flow.
	\end{theorem}
	\begin{proof}We  show a slightly stronger result that \eqref{qts} follows from \eqref{proflow} and \eqref{prolin} used for
		$s=0$ only. For $t>0$, denote $\bh_t={\bf x}_{0t}$ and let $\bh_0=\mathbf{e}$. By assumption,
		$\bh_t^{-1}$ exists, so from \eqref{proflow} we see that for $0\leq t< u$ we have
		\begin{equation}\label{pre-qts}
			{\bf x}_{tu}=\bh_t^{-1}\boxdot \bh_u.
		\end{equation}
		Inserting $\bh_u$ into  \eqref{prolin}  with $s=0$, we see that there exists $\bh$ such that
		$
		\frac{1}{u}(\bh_u-\mathbf{e})=\bh
		$, i.e., $\bh_u=\mathbf{e}+u\bh$. In particular,  $\mathbf{e}+u\bh$ is invertible for all $u\geq 0$ and \eqref{qts} follows
		from \eqref{pre-qts}.
		
		To prove uniqueness, suppose \eqref{qts} holds with $\bh$ and $\bh'$, i.e., $(\mathbf{e}+s\bh)^{-1}\boxdot
		(\mathbf{e}+t\bh)=(\mathbf{e}+s\bh')^{-1}\boxdot (\mathbf{e}+t\bh')$. Taking $s=0$ gives $\bh=\bh'$.
		
		To prove the converse, it is clear that expression \eqref{qts} solves \eqref{proflow} for all $0\leq s<t<u$
		and that, as a product of invertible elements, ${\bf x}_{st}$ is invertible.
		To verify \eqref{prolin}, we write its left hand side as
		\small
		\begin{align*}
			\frac{1}{u-s}\left((\mathbf{e}+s\bh)^{-1}\boxdot(\mathbf{e}+u\bh)-\mathbf{e}\right)&=
			\frac{1}{u-s}(\mathbf{e}+s\bh)^{-1}\boxdot\left((\mathbf{e}+u\bh)-(\mathbf{e}+s\bh)\right)\\
			&=(\mathbf{e}+s\bh)^{-1}\boxdot\bh
		\end{align*}
		\normalsize
		which does not depend on $u$.
	\end{proof}

	\subsection{Flows in the near algebra of linear maps.}\label{Ex_of_DNA}
	
	Let $V$ be a linear space over $\R$. For any  ${\bf a}=(\alpha,\ua),\,{\bf b}=(\beta,\ub)\in \R\times V$ define
	\begin{equation}
		\label{QHproduct}
		{\bf a}\boxdot{\bf b}=(\beta\alpha,\,\beta \ua+\ub).
	\end{equation}
	One can check that  $\mathcal A_V:=(\R\times V,+,\boxdot)$ is a near algebra with $\mathbf e=(1,\underline 0)$.
	
	We observe that any ${\bf a}=(\alpha,\ua)\in\R\times V$ with $\alpha\neq 0$ is invertible with the inverse  %
	$${\bf a}^{-1}=(\alpha^{-1},\,-\alpha^{-1}\,\ua).$$
	Non-invertible ${\bf a}=(\alpha,\ua)\in \R\times V$  have   $\alpha=0$, so are  null, as \eqref{QHproduct} gives
	\eqref{A6}.
	
	Note that $\mathcal A_V$ can be  viewed as algebra of linear maps
	$f_{\bf a}:\R\times V^*\to \R$, where $V^*$ is the algebraic dual of $V$, ${\bf a}\in\R\times V$. For ${\bf
		a}=(\alpha,\ua)\in\R\times V$ the map $f_{\bf a}$ is  defined by
	$$f_{\bf a}(\lambda,\Lambda)=\lambda\alpha+\Lambda\,\ua,\quad (\lambda,\Lambda)\in\R\times V^*,$$  Then, for ${\bf
		a},\,{\bf b}\in \R\times V$ we define a composition $f_{\bf a}\circ f_{\bf b}$ by $$(f_{\bf a}\circ f_{\bf
		b})(\lambda,\Lambda):=f_{\bf b}(f_{\bf a}(\lambda,\Lambda),\Lambda),\quad (\lambda,\Lambda)\in \R\times V^*.$$
	This composition corresponds to operation \eqref{QHproduct} on pairs ${\bf a}, {\bf b}$, i.e.
	$$f_{\bf a}\circ f_{\bf b}=f_{{\bf a}\boxdot{\bf b}}.$$
	
	\begin{proposition}\label{corI}
		Let  ${\bf x}_{st}=\left(\xi_{st},\underline x_{st}\right)\in\R\times V$, $0\le s<t$, be such that $\xi_{st}\ne 0$
		for all $0\leq s<t$.
		Assume that \eqref{proflow} and  \eqref{prolin}  are satisfied.
		Then there exist $\alpha\ge 0$ and $\ua\in V$ such that for any $0\le s<t$
		\begin{equation}\label{22}
			{\bf x}_{st}=\left(\tfrac{1+\alpha t}{1+\alpha s},\,\tfrac{t-s}{1+\alpha s}\,\ua\right).
		\end{equation}
	\end{proposition}
	\begin{proof}
		We apply Theorem \ref{protI} with $\bh=(\alpha,\ua)$. It is clear
		that  ${\bf e}+s\bh$
		is invertible for all $s\geq 0$ if and only if $\alpha\geq 0$. Formula \eqref{22} follows by calculation.
	\end{proof}

	\begin{remark}\label{remR2}
		If we assume that the flow equation \eqref{proflow} and  condition \eqref{prolin} are satisfied only for $0<s<t<u$
		then, as shown in the proof of Proposition \ref{C.2.3}, additional solutions arise.

	\end{remark}

	Another example is related to endomorphisms. Let $L:=L(V)$ be a space of endomorphisms of a linear space $V$. For any
	${\bf A}=(A_1,A_2),\,{\bf B}=(B_1,B_2)\in L^2$ define
	$$
	{\bf A}\boxdot {\bf B}=(A_1B_1,\,A_2B_1+B_2).
	$$
	One can easily check that $(L^2,+,\boxdot)$ is a near algebra with $\mathbf e=(\mathrm{Id},0)$.
	Any ${\bf A}=(A_1,A_2)\in L^2$ with invertible $A_1$ has inverse of the form
	$$
	{\bf A}^{-\boxdot}=(A_1^{-1},-A_2A_1^{-1}).
	$$
	Elements of $L^2$ with $A_1=0$  are null and they form a proper subclass of non-invertible elements of this near-algebra.
	Application of Theorem \ref{protI} gives the solution of the flow equation for a family ${\bf x}_{st}$, $0\le s<t$, of
	invertible elements satisfying  condition \eqref{prolin}:
	
	There exist $G,\,H\in L$,  where $\mathrm{Id}+sH$ is invertible for all  $s\geq 0$ and
	$$
	{\bf x}_{st}=\left((\mathrm{Id}+tH)(\mathrm{Id}+sH)^{-1},\,(t-s)G(\mathrm{Id}+sH)^{-1}\right),\quad 0\le s<t.
	$$
	We will consider a related example more thoroughly, while discussing double near algebras in Section \ref{DNAofAF}.

	\section{Flows in double near algebras}\label{EDNA}
	We now introduce our  main algebraic object.
	\begin{definition}\label{Def:EDNA}
		We say that $(V,+,\ltimes,\rtimes)$ is a {\em double near algebra} (DNA) if
		\begin{enumerate}
			\item $(V,+,\ltimes)$ and $(V,+,\rtimes)$ are near algebras with neutral elements $\mathbf{e}_\ltimes$ and
			$\mathbf{e}_\rtimes$, respectively,
			\item  $\mathbf{e}_\ltimes$  is $\rtimes$-null and $\mathbf{e}_\rtimes$ is $\ltimes$-null, i.e.,
			for every ${\bf x} \in V$
			\begin{equation}
				{\bf x}\rtimes \mathbf{e}_\ltimes=\mathbf{e}_\ltimes \quad \textnormal{and} \quad {\bf x}\ltimes
				\mathbf{e}_\rtimes=\mathbf{e}_\rtimes.
				\label{mult_by_neutr}
			\end{equation}
		\end{enumerate}
	\end{definition}
	We  denote by ${\bf x}^{-\ltimes}$ and ${\bf x}^{-\rtimes}$ the inverse elements of ${\bf x}$ with respect to
	multiplications $\ltimes$
	and $\rtimes$, respectively, if they exist.

	\begin{definition}
		A family   $({\bf x}_{sru})_{0\le r\le s<u}$ \hspace{1mm}  of  $\ltimes$- and $\rtimes$-invertible elements  of
		$V$   is called a
		{\em two-way flow},
		if it satisfies the flow equations
		\begin{equation}\label{RL_mult}
			{\bf x}_{sru}\rtimes {\bf x}_{tsu}={\bf x}_{tru}, \quad
			{\bf x}_{tru}\ltimes {\bf x}_{srt}={\bf x}_{sru}, \qquad 0\leq r<s<t<u,
		\end{equation}
		and the structure condition: for all $0\le r<u$
		\begin{equation}
			\tfrac{{\bf x}_{sru}-\mathbf{e}_\ltimes}{u-s}+\tfrac{{\bf x}_{sru}-\mathbf{e}_\rtimes}{s-r}
			\mbox{ does not depend on $s\in(r,u)$.}
			\label{S-C}
		\end{equation}
		
	\end{definition}
	
	Note that the family \begin{equation}\label{triv}{\bf x}_{sru}=\tfrac{u-s}{u-r}{\bf e}_{\rtimes}+\tfrac{s-r}{u-r}{\bf
			e}_{\ltimes},\quad 0\le r<s<u,\end{equation} is a two-way flow in any DNA. To see $\ltimes$- and $\rtimes$-invertibility
	we use \eqref{xf}. Flow equations follow by straightforward calculation. The expression in  \eqref{S-C} is  zero.

	We will use the following notation. For any $\bh\in V$ and $r\geq 0$ denote
	\begin{equation}\label{hr}\bH_r =	r(1-r)\bh+(1-r)\mathbf{e}_\rtimes+r\mathbf{e}_\ltimes.\end{equation}
	
	The following result is a parametric description of all two-way flows.
	\begin{theorem}	\label{Tw o  h}
		If a family $({\bf x}_{sru})_{0\leq r<s<u}$  is a two-way flow, then there exists unique   $\bh\in V$ such that for all
		$0\le r<s<u$
		\begin{equation}	\label{xxsru_def}
			{\bf x}_{sru}
			=(\bh_u^{-\ltimes}\ltimes\bh_r)^{-\rtimes}\rtimes(\bh_u^{-\ltimes}\ltimes\bh_s).
		\end{equation}
	\end{theorem}
	For later  reference it will be convenient to rewrite \eqref{xxsru_def} as
	\begin{equation}	\label{xsru_def}
		{\bf x}_{sru}
		=  \bw_{ru}^{-\oR} \oR \bw_{su},
	\end{equation}
	where for $0\leq r <u$
	\begin{equation}
		\bw_{ru}=\bh_u^{-\oL} \oL \bh_r.
		\label{definicjahs}
	\end{equation}
	We remark that $\bw_{st}:=\mathbf{x}_{s0t} $ solves induced flow equation $\bw_{tu}\ltimes \bw_{st}=\bw_{su}$  but
	although
	its solution \eqref{definicjahs}
	looks deceptively similar to  \eqref{qts},
	the simple argument we used to derive \eqref{qts} is not applicable here, as in our solution $\bw_{0u}=\mathbf{e}_\rtimes$
	is not $\ltimes$-invertible.

	A computation shows that \eqref{hr} and \eqref{xxsru_def} with  $\bh=0$ give the   two-way flow \eqref{triv}.
	
	\begin{proof}
		Denote $${\bf w}_{su}:= \begin{cases}
			{\bf x}_{s0u} & \mbox{for $0<s<u,$}\\
			{\bf e}_\rtimes& \mbox{for $0=s<u.$}
		\end{cases}$$
		From the first flow equation in \eqref{RL_mult} we have
		\begin{equation}\label{rown_1_AZ}
			{\bf x}_{sru}={\bf w}_{ru}^{-\rtimes}\rtimes {\bf w}_{su}
		\end{equation}
		for all $0\leq r <s<u$.
		Furthermore, \begin{equation}\label{rown_2_AZ}
			{\bf w}_{su} = {\bf v}_u^{-\ltimes}\ltimes{\bf v}_s
		\end{equation} for all $0\le s<u$, where
		$${\bf v} _{s} := \left\{ \begin{array}{ll}
			\bw_{1s}^{-\ltimes} & \textrm{for $s>1,$}\\
			{\bf e}_\ltimes& \textrm{for $s=1,$}\\
			\bw_{s1}& \textrm{for $0<s<1,$}\\
			{\bf e}_{\rtimes}& \textrm{for $s=0,$}
		\end{array} \right.$$
		where the case $0<s<u$ follows from the second equation in \eqref{RL_mult} and the case $s=0<u$ is a simple consequence of
		\eqref{mult_by_neutr}.
		So from the above considerations it is sufficient to give only a parametrization of ${\bf x}_{s01}$ and ${\bf x}_{10s}$ for
		$s<1$ and $s>1$ respectively. Firstly, note that after a simple rewrite of \eqref{S-C}, we get that there exists
		$\bg_{u}\in V$ such that
		\begin{equation}\label{def gu}
			\bg_{u}:=\tfrac{u}{(u-t)t}{\bf x}_{t0u}-\tfrac{1}{u-t}\mathbf{e}_\ltimes-\tfrac{1}{t}\mathbf{e}_\rtimes\in V.
		\end{equation}
		Thus for all $0<t<u$ we obtain
		\begin{equation}	\label{xs0u from gu}
			{\bf x}_{t0u}=\tfrac{(u-t)t}{u}\bg_{u}+\tfrac{t}{u}\mathbf{e}_\ltimes+\tfrac{u-t}{u}\mathbf{e}_\rtimes.
		\end{equation}
		Inserting expression \eqref{xs0u from gu} for ${\bf x}_{s0t}$ and ${\bf x}_{s0u}$ into the second equation in
		(\ref{RL_mult}) as well as using (\ref{ml}), (\ref{L-jednorodn}) and (\ref{mult_by_neutr}) gives
		$$\tfrac{(t-s)s}{t} {\bf x}_{t0u}\ltimes {\bf g}_t+\tfrac{t-s}{t}{\bf e}_\rtimes+\tfrac{s}{t} {\bf x}
		_{t0u}=\tfrac{s(u-s)}{u}{\bf g}_u+\tfrac{u-s}{u}{\bf e}_\rtimes+\tfrac{s}{u}{\bf e}_\ltimes.$$
		Substituting ${\bf x}_{t0u}$ in the third component at the left hand side above according to (\ref{xs0u from gu}),
		after a simple calculation, we get $$\tfrac{1}{t}{\bf x}_{t0u}\ltimes \bg_{t}=\tfrac{1}{u}\bg_{u},\qquad 0<t<u.$$ Using
		(\ref{ml}), (\ref{L-jednorodn}), (\ref{mult_by_neutr}), (\ref{def gu}) and the assumption of   $\ltimes$-invertibility of
		${\bf x}_{t0u}$ we obtain that
		$${\bf x}_{t0u}\ltimes \left(\tfrac{1}{t}\bg_{t}-\tfrac{1}{(u-t)t}\mathbf{e}_\ltimes+\tfrac{1}{u(u-t)}{\bf
			x}_{t0u}^{-\ltimes}+\tfrac{1}{tu}\mathbf{e}_\rtimes\right)={\bf 0}\in V.$$
		Applying  ${\bf x}_{t0u}^{-\ltimes}\ltimes$ to the above (recall that ${\bf 0}$ is  $\ltimes$-null) we get
		$$
		\tfrac{1}{t}\bg_{t}-\tfrac{1}{(u-t)t}\mathbf{e}_\ltimes+\tfrac{1}{u(u-t)}{\bf
			x}_{t0u}^{-\ltimes}+\tfrac{1}{tu}\mathbf{e}_\rtimes={\bf 0}.
		$$
		Therefore, for all $0 <t<u$ we have
		\begin{equation}	\label{def-x0u-ltimes}
			{\bf x}_{t0u}^{-\ltimes}=-\tfrac{(u-t)u}{t}{\bf g}_t+\tfrac{u}{t}{\bf e}_\ltimes-\tfrac{u-t}{t}{\bf e}_\rtimes.
		\end{equation}
		Let ${\bf h}:={\bf  g}_1$. Then \eqref{xs0u from gu} gives
		$${\bf x}_{t01}=t(1-t){\bf h} + t{\bf e}_{\ltimes}+(1-t){\bf e}_\rtimes={\bf h}_t,\quad t<1,$$
		where the last equation holds due to definition of $\bf h_t$. Moreover, from \eqref{def-x0u-ltimes} we get
		$${\bf x}_{10u}^{-\ltimes}=u(1-u){\bf  h}+u{\bf e}_\ltimes+(1-u){\bf e}_\rtimes={\bf h}_u,\quad u>1.$$
		To sum up, we have shown that ${\bf v}_s={\bf h}_s$ for all $s\geq 0$. This identity through \eqref{rown_1_AZ} and
		\eqref{rown_2_AZ} gives the conclusion \eqref{xxsru_def}.
		
		To see that $\bh$ is unique, assume that there exists $\widetilde{\bh}\neq \bh$ such that plugging $\widetilde{\bh}$
		instead of $\bh$ in \eqref{hr} we obtain $\widetilde{\bh}_r$ and then through \eqref{xxsru_def} another expression for
		${\bf x}_{sru}$. Then, comparing these two expressions for ${\bf x}_{sru}$ we get
		$$
		(\widetilde{\bh}_u^{-\ltimes}\ltimes\widetilde{\bh}_r)^{-\rtimes}\rtimes(\widetilde{\bh}_u^{-\ltimes}\ltimes\widetilde{\bh}_s)=(\bh_u^{-\ltimes}\ltimes\bh_r)^{-\rtimes}\rtimes(\bh_u^{-\ltimes}\ltimes\bh_s),\quad
		0\le r<s<u.
		$$
		Since $\widetilde{\bh}_0=\bh_0={\bf e}_{\rtimes}$ the above equation with $r=0$ yields
		$$
		\widetilde{\bh}_u^{-\ltimes}\ltimes\widetilde{\bh}_s=\bh_u^{-\ltimes}\ltimes\bh_s,\quad 0<s<u,
		$$
		or equivalently
		$$\bh_u\ltimes\widetilde{\bh}_u^{-\ltimes}=\bh_s\ltimes\widetilde{\bh}_s^{-\ltimes}.$$
		Thus $\bh_u\ltimes\widetilde{\bh}_u^{-\ltimes}$ does not depend on $u$. Since $\bh_1=\widetilde{\bh}_1={\bf
			e}_{\ltimes}$ it follows that $\bh_u\ltimes\widetilde{\bh}_u^{-\ltimes}={\bf e}_\ltimes$ for all $u>0$ and consequently
		$\bh=\widetilde{\bh}$.
	\end{proof}

	\begin{remark}\label{minim}
		It follows from the proof above that to get representations \eqref{hr}, \eqref{xxsru_def}, it suffices to assume that
		\begin{enumerate}
			\item ${\bf x}_{s0u}$ is $\ltimes$- and $\rtimes$-invertible, $0<s<u$,
			\item $({\bf x}_{s0u})_{0<s<u}$  satisfies \eqref{S-C} and the second flow equation \eqref{RL_mult} for $r=0$,
			\item $({\bf x}_{sru})_{0\leq r<s<u}$  satisfies the first flow equation \eqref{RL_mult}.
		\end{enumerate}
		
		Note also that in view of (i) it follows from (iii) that ${\bf x}_{sru}$ is $\rtimes$-invertible for any $0\le r<s<u$.
	\end{remark}
	Now we will focus on a converse to Theorem \ref{Tw o  h}, i.e., we will study the following question: for which $\bh\in V$ the family given by \eqref{xxsru_def} is a two-way flow.  To do that we need some identities for the elements $\bh_r$, $r\geq0$.
	\begin{lemma}
		Let $0\le r<s<u$. The following identities hold:
		\begin{equation}\label{h_ident}
			\bh_s=\tfrac{u-s}{u-r}\bh_r+\tfrac{s-r}{u-r}\bh_u+(u-s)(s-r)\bh,
		\end{equation}
		\begin{equation}\label{h_ident1}
			ru\, \bh_s=\tfrac{u-s}{u-r}\, su\, \bh_r+\tfrac{s-r}{u-r}\, rs\, \bh_u-(u-s)(s-r){\bf e}_{\rtimes},
		\end{equation}
		\begin{equation}\label{hrstu}
			\tfrac{\bh_u}{(u-t)(u-s)(u-r)}+\tfrac{\bh_s}{(u-s)(t-s)(s-r)}=\tfrac{\bh_r}{(u-r)(t-r)(s-r)}+\tfrac{\bh_t}{(u-t)(t-s)(t-r)}.
		\end{equation}
	\end{lemma}
	\begin{proof}
		All these identities are easily obtained by direct calculations based on the coefficients in \eqref{hr}.
	\end{proof}
	We now give a partial converse to Theorem \ref{Tw o  h}, which implies a converse to Remark \ref{minim}.
	\begin{proposition}
		\label{P-left-C}
		If $\bh\in V$ is such that the right hand side of \eqref{xxsru_def} is well defined then family $({\bf x}_{sru})_{0\leq
			r<s<u}$   defined by \eqref{hr} and \eqref{xxsru_def}  	satisfies the structure condition \eqref{S-C}, is
		$\rtimes$-invertible  and
		the first flow equation \eqref{RL_mult} holds.
		In addition, ${\bf x}_{s0u}$  is $\ltimes$-invertible, and the second flow equation \eqref{RL_mult}  holds for
		$r=0<s<t<u$.
		
	\end{proposition}
	\begin{proof}
		As a product \eqref{xxsru_def} of $\rtimes$-invertible elements, ${\bf x}_{sru}
		$ is  $\rtimes$-invertible for $r\geq 0$. Furthermore, by associativity of $\rtimes$ the first equation of \eqref{RL_mult}
		holds. Moreover,  \eqref{xsru_def} gives ${\bf x}_{s0u}={\bf w}_{su}$, which is  $\ltimes$-invertible from
		\eqref{definicjahs}, and associativity of $\ltimes$ gives the second equation of \eqref{RL_mult} for $r=0$. To verify
		structure condition \eqref{S-C} note that
		\begin{equation*}
			\begin{split}
				\tfrac{{\bf x}_{sru}-\mathbf{e}_\ltimes}{u-s}+\tfrac{{\bf x}_{sru}-\mathbf{e}_\rtimes}{s-r}&={\bf
					w}_{ru}^{-\rtimes}\rtimes\big(\tfrac{u-r}{(u-s)(s-r)}{\bf w}_{su}-\tfrac{1}{u-s}{\bf e}_{\ltimes}-\tfrac{1}{s-r}{\bf
					w}_{ru}\big)\\
				&={\bf w}_{ru}^{-\rtimes}\rtimes\big\{{\bf h}_{u}^{-\ltimes}\ltimes\big[\tfrac{u-r}{(u-s)(s-r)}{\bf
					h}_{s}-\tfrac{1}{u-s}{\bf h}_{u}-\tfrac{1}{s-r}{\bf h}_{r}\big]\big\}\\
				&={\bf w}_{ru}^{-\rtimes}\rtimes\big\{{\bf h}_{u}^{-\ltimes}\ltimes((u-r){\bf h})\big\},
			\end{split}
		\end{equation*}
		where the last equality follows from  \eqref{h_ident}. Thus \eqref{S-C} is satisfied for all $0\leq r<s<u$.
	\end{proof}

	As Example \ref{Kontrprzyklad} shows, full converse of Theorem \ref{Tw o  h} requires additional assumptions.  We introduce
	the following concept.
	\begin{definition} \label{def-gen}
		We   say that $\bh\in V$ is  a {\em flow generator}
		if
		$\bH_r$ in \eqref{hr} is $\ltimes$-invertible for every $r>0$ and
		defines by \eqref{definicjahs} the $\rtimes$-invertible family $({\bf w}_{ru})_{0<r<u}$
		satisfying for all $0<r<s<u$
		\begin{equation}\label{generator}
			\bw_{ru}^{-\oR} \oR \bw_{su}=(\bw_{rs}^{-\oR} \oR \bw_{su}^{-\oL})^{-\oL}\,.
		\end{equation}	
		(In particular, we assume that  the inverse on the right hand side of \eqref{generator}   exists.)	
	\end{definition}
	
	Since this definition is  rather cumbersome, we give a somewhat simpler  sufficient condition.
	\begin{proposition}\label{P-suff}
		Suppose that $\bh\in V$  is such that expression (\ref{definicjahs}) is well defined and the corresponding   family  $({\bf
			w}_{ru})_{0<r<u}$     is such that   for all $0<r<u$ expression
		\begin{equation} \label{sufficiency}
			\tfrac{1}{u(u-r)}\big({\bf
				w}_{ru}^{-\rtimes}\big)^{-\ltimes}+\tfrac{1}{ru}\mathbf{e}_\ltimes-\tfrac{1}{r(u-r)}\mathbf{e}_\rtimes
		\end{equation}
		is well defined and does not depend on $u$. Then $\bh$ is a flow generator. %
	\end{proposition}
	\begin{proof} Since \eqref{sufficiency} does not depend on $u$  for all $0<r<s<u$ we have
		$$\tfrac{1}{u(u-r)}({\bf
			w}_{ru}^{-\rtimes})^{-\ltimes}+\tfrac{1}{ru}\mathbf{e}_\ltimes-\tfrac{1}{r(u-r)}\mathbf{e}_\rtimes= \tfrac{1}{s(s-r)}({\bf
			w}_{rs}^{-\rtimes})^{-\ltimes}+\tfrac{1}{rs}\mathbf{e}_\ltimes-\tfrac{1}{r(s-r)}\mathbf{e}_\rtimes.$$
		Consequently,   solving this for $({\bf w}_{rs}^{-\rtimes})^{-\ltimes}$  by straightforward algebra we have
		$$
		{\bf w}_{rs}^{-\rtimes}=\left(\tfrac{s(s-r)}{u(u-r)}{ \left({\bf
				w}_{ru}^{-\rtimes}\right)^{-\ltimes}}-\tfrac{(u-s)(s-r)}{ru}{\bf e}_{\ltimes}+\tfrac{s(u-s)}{r(u-r)}{\bf
			e}_{\rtimes}\right)^{-\ltimes}.$$
		Using \eqref{Jacek} with ${\bf x}={\bf
			w}_{ru}^{-\rtimes}$, ${\bf e}={\bf e}_{\ltimes}$, ${\bf f}={\bf e}_{\rtimes}$, $\alpha=-\tfrac{(u-s)(s-r)}{ru}$, $\beta=\tfrac{s(s-r)}{u(u-r)}$ and $\gamma=\tfrac{s(u-s)}{r(u-r)}$ we get
		$$
		\tfrac{(u-r)(u-s)}{rs}\bw_{rs}^{-\rtimes}=\left(\tfrac{s(s-r)}{u(u-r)}{\bf e}_{\ltimes}+\tfrac{s(u-s)}{r(u-r)}{\bf
			e}_{\rtimes}-\tfrac{(u-s)(s-r)}{ru}\bw_{ru}^{-\rtimes}\right)^{-\ltimes}-\tfrac{u(u-r)}{s(s-r)}{\bf
			e}_{\ltimes}+\tfrac{u(u-s)}{r(s-r)}{\bf e}_{\rtimes}.
		$$
		So %
		\begin{align} \label{rown_3_AZ}
			&\bw_{rs}^{-\rtimes}\rtimes\left(\tfrac{(u-r)(u-s)}{rs}{\bf e}_{\rtimes}+\tfrac{u(u-r)}{s(s-r)}{\bf
				e}_{\ltimes}-\tfrac{u(u-s)}{r(s-r)}\bw_{rs}\right)\\
			&=\left(\bw_{ru}^{-\rtimes}\rtimes\left(\tfrac{s(s-r)}{u(u-r)}{\bf e}_{\ltimes}-\tfrac{(u-s)(s-r)}{ru}{\bf
				e}_{\rtimes}+\-\tfrac{s(u-s)}{r(u-r)}\bw_{ru}\right)\right)^{-\ltimes}.\nonumber
		\end{align}
		Next we note that since $\mathbf w_{rs}=\bh_s^{-\ltimes}\ltimes\bh_r$, see \eqref{definicjahs}, and ${\bf e}_{\rtimes}$
		is $\ltimes$-null, the expression in the parentheses on the left hand side of  \eqref{rown_3_AZ} can be rewritten as
		\begin{align*}
			&\tfrac{(u-r)(u-s)}{rs}{\bf e}_{\rtimes}+\tfrac{u(u-r)}{s(s-r)}{\bf e}_{\ltimes}-\tfrac{u(u-s)}{r(s-r)}\bw_{rs}\\
			&=\bh_s^{-\ltimes}\ltimes\left(\tfrac{u(u-r)}{s(s-r)}\bh_s-\tfrac{u(u-s)}{r(s-r)}\bh_r+\tfrac{(u-r)(u-s)}{rs}{\bf
				e}_{\rtimes}\right) \stackrel{\eqref{h_ident1}}{=}\bh_s^{-\ltimes}\ltimes\bh_u=\bw_{su}^{-\ltimes}.
		\end{align*}
		Analogously, we  rewrite the expression in the inner parentheses on the right hand side of \eqref{rown_3_AZ}. We get
		$$
		\tfrac{s(s-r)}{u(u-r)}{\bf e}_{\ltimes}-\tfrac{(u-s)(s-r)}{ru}{\bf e}_{\rtimes}+\-\tfrac{s(u-s)}{r(u-r)}\bw_{ru}=\bw_{su}.
		$$
		Thus \eqref{rown_3_AZ}  simplifies to
		\[
		\bw_{rs}^{-\rtimes}\rtimes\bw_{su}^{-\ltimes}=\left(\bw_{ru}^{-\rtimes}\rtimes\bw_{su}\right)^{-\ltimes}.
		\]
		Taking the $\ltimes$-inverse yields \eqref{generator}.
	\end{proof}
	
	\begin{theorem}	\label{T-flow-exists}
		Let $\bh\in V$ and let $({\bf x}_{sru})_{0\leq r<s<u}$ be  generated  by $\bh$ according to \eqref{hr} and
		\eqref{xxsru_def}. The family  $({\bf x}_{sru})_{0\leq r<s<u}$  is a two-way flow if and only if  $\bh$ is a flow
		generator.
	\end{theorem}
	
	\begin{proof}
		{\em Sufficiency:} If $\bh\in V$ is a flow generator, then  in view of Proposition \ref{P-suff}, it remains only to verify
		$\ltimes$-invertibility and the second flow equation for $r>0$.
		We see that the left hand side of \eqref{generator}  has $\ltimes$-inverse, so (\ref{xsru_def}) implies that
		${\bf x}_{sru}$ is   $\ltimes$-invertible for $0<r<s<u$ with the inverse
		\begin{equation}\label{xsru-inv}
			{\bf x}_{sru}^{-\ltimes}={\bf w}_{rs}^{-\rtimes}\rtimes{\bf w}_{su}^{-\ltimes}.
		\end{equation}

		We rewrite the structure condition \eqref{S-C} (we know { from} Proposition \ref{P-left-C} that it holds) as
		\begin{equation}\label{rown2}
			\tfrac{u-t}{u-s}{\bf
				x}_{sru}+\tfrac{(t-s)(s-r)}{(u-s)(u-r)}\mathbf{e}_\ltimes=\tfrac{(t-s)(u-t)}{(t-r)(u-r)}\mathbf{e}_\rtimes+\tfrac{s-r}{t-r}{\bf
				x}_{tru}
		\end{equation}
		for $0<r<s<t<u$.
		Combining this with (\ref{xsru-inv}) gives
		
		\begin{equation*}
			\begin{split}
				{\bf x}_{sru}^{-\ltimes}\ltimes {\bf x}_{tru}&=\tfrac{(t-s)(t-r)}{(u-s)(u-r)}{\bf x}_{sru}^{-\ltimes}+
				\tfrac{(t-r)(u-t)}{(s-r)(u-s)}\mathbf{e}_\ltimes-\tfrac{(t-s)(u-t)}{(s-r)(u-r)}\mathbf{e}_\rtimes\\
				&={\bf w}_{rs}^{-\rtimes}\rtimes\bigg(\tfrac{(t-s)(t-r)}{(u-s)(u-r)}{\bf w}_{su}^{-\ltimes}+
				\tfrac{(t-r)(u-t)}{(s-r)(u-s)}\mathbf{e}_\ltimes-\tfrac{(t-s)(u-t)}{(s-r)(u-r)}\mathbf{w}_{rs}\bigg)\\
				&={\bf w}_{rs}^{-\rtimes}\rtimes\left[{\bf
					h}_{s}^{-\ltimes}\ltimes\left(\tfrac{(t-s)(t-r)}{(u-s)(u-r)}\bh_u+\tfrac{(t-r)(u-t)}{(s-r)(u-s)}\bh_s-\tfrac{(t-s)(u-t)}{(s-r)(u-r)}\bh_r\right)\right].
			\end{split}
		\end{equation*}
		Thus  \eqref{hrstu}  yields
		$$ {\bf x}_{sru}^{-\ltimes}\ltimes {\bf x}_{tru}={\bf w}_{rs}^{-\rtimes}\rtimes\big({\bf h}_{s}^{-\ltimes}\ltimes{\bf
			h}_{t}\big)={\bf x}_{srt}^{-\ltimes},$$
		whence the second flow equation (\ref{RL_mult}) for $r>0$ follows.
		
		{\em  Necessity:} If $({\bf x}_{sru})_{0\leq r<s<u}$ is a two-way flow, then as in the proof of Prop. \ref{P-left-C}, we
		get
		\begin{equation}\label{cotojestx_sru}
			{\bf x}_{sru}=(u-s)(s-r){\bf w}_{ru}^{-\rtimes} \rtimes ({\bf h}_u^{-\ltimes}\ltimes {\bf h})+\tfrac{s-r}{u-r}{\bf
				e}_\ltimes+\tfrac{u-s}{u-r}{\bf e}_\rtimes.
		\end{equation}
		Inserting this formula for ${\bf x}_{srt}$ into the second flow equation yields
		\begin{equation*}
			{\bf x}_{sru}=(t-s)(s-r){\bf x}_{tru}\ltimes\big\{{\bf w}_{rt}^{-\rtimes} \rtimes ({\bf h}_t^{-\ltimes}\ltimes {\bf
				h})\big\}+\tfrac{s-r}{t-r}{\bf x}_{tru} +\tfrac{t-s}{t-r}{\bf e}_\rtimes.
		\end{equation*}
		Substituting ${\bf x}_{sru}$ and the second ${\bf x}_{tru}$ above according to \eqref{cotojestx_sru}, after
		simplifications,  we obtain
		\begin{align*}{\bf x}_{tru}\ltimes\big\{{\bf w}_{rt}^{-\rtimes} \rtimes ({\bf h}_t^{-\ltimes}\ltimes {\bf h})\big\}&={\bf
				w}_{ru}^{-\rtimes} \rtimes ({\bf h}_u^{-\ltimes}\ltimes {\bf h})\\
			&=\tfrac{1}{(u-t)(t-r)}{\bf x}_{tru}-\tfrac{1}{(u-t)(u-r)}{\bf e}_\ltimes -\tfrac{1}{(t-r)(u-r)}{\bf e}_\rtimes,
		\end{align*}
		where the last equality holds due to \eqref{cotojestx_sru}. Applying ${\bf x}_{tru}^{-\ltimes}\ltimes$ to equality of the
		first and the third expression above and solving for ${\bf x}_{tru}^{-\ltimes}$ we get
		\begin{equation*}
			\begin{split}
				{\bf x}_{tru}^{-\ltimes}&=\tfrac{u-r}{t-r}{\bf e}_\ltimes-\tfrac{u-t}{t-r}{\bf e}_\rtimes-(u-t)(u-r){\bf
					w}_{rt}^{-\rtimes} \rtimes ({\bf h}_t^{-\ltimes}\ltimes {\bf h})\\
				&={\bf w}_{rt}^{-\rtimes}\rtimes \big(\tfrac{u-r}{t-r}{\bf e}_\ltimes-\tfrac{u-t}{t-r}{\bf w}_{rt}-(u-t)(u-r){\bf
					h}_t^{-\ltimes}\ltimes {\bf h}\big)\\
				&={\bf w}_{rt}^{-\rtimes}\rtimes
				\big[\bh_t^{-\ltimes}\ltimes\big(\tfrac{u-r}{t-r}\bh_t-\tfrac{u-t}{t-r}\bh_r-(u-t)(u-r)\bh\big)\big].
			\end{split}
		\end{equation*}
		Thus, by \eqref{h_ident},
		$$
		{\bf x}_{tru}^{-\ltimes}={\bf w}_{rt}^{-\rtimes}\rtimes \big({\bf h}_t^{-\ltimes}\ltimes {\bf h}_u\big)={\bf
			w}_{rt}^{-\rtimes}\rtimes{\bf w}_{tu}^{-\ltimes}.
		$$
		Recalling that ${\bf x}_{tru}={\bf w}_{ru}^{-\rtimes} \rtimes{\bf w}_{tu}$, this gives \eqref{generator}, so $\bf h$ is a
		flow generator.
	\end{proof}
	
	\subsection{Two way flows in the  DNA of linear maps}
	\label{DNAofAF}
	Let  $L=L(V)$ be a space of endomorphisms on a linear space $V$. For $A_1,A_2,B_1,B_2\in L$ we define
	$$(A_1, A_2) \ltimes (B_1, B_2) = (B_1 + A_1 B_2, A_2 B_2)$$
	and
	$$(A_1, A_2) \rtimes (B_1, B_2) = (A_1 B_1, A_2 B_1 + B_2).$$
	
	Then $\mathcal L_V:=(L^2,+,\ltimes,\rtimes)$ is a DNA with ${\bf e}_{\ltimes}=(\mathrm{0},\mathrm{Id})$ and ${\bf
		e}_{\rtimes}=(\mathrm{Id},0)$.
	Let ${\bf A}=(A_1,A_2)$ be an element of this DNA.
	When $A_2$ is invertible then ${\bf A}^{-\ltimes}=(-A_1A_2^{-1},A_2^{-1})$ exists. Similarly, when $A_1$ is invertible then
	${\bf A}^{-\rtimes}=(A_1^{-1},-A_2A_1^{-1})$.
	
	This DNA can be interpreted as the family of linear  maps  $f_{\bf A}:L^2\to L$ with suitable compositions. For ${\bf
		A}=(A_1,A_2)\in L^2$  we define
	$$
	f_{\bf A} (X,Y) =XA_1+Y A_2,\quad X,Y\in L.
	$$
	For such family of maps, $\{f_{\bf A},\,{\bf A}\in L^2\}$, there are two natural compositions
	\small
	$$(f_{\bf A}\circ_1 f_{\bf B})(X,Y):=f_{\bf B}(f_{\bf A}(X,Y),Y)\quad \mbox{and}\quad  (f_{\bf A}\circ_2 f_{\bf
		B})(X,Y):=f_{\bf B}(X,f_{\bf A}(X,Y)).$$
		\normalsize
	Note that they are directly related to the products $\rtimes$ and $\ltimes$ in $\mathcal L_V$:
	$$
	f_{\bf A}\circ_1 f_{\bf B}=f_{\bf A\rtimes\bf B}\quad \mbox{and}\quad f_{\bf A}\circ_2 f_{\bf B}=f_{\bf A\ltimes\bf B}.
	$$
	
	We now apply Theorem \ref{Tw o  h} in this setting.
	\begin{proposition}\label{P-p}
		
		The following conditions are equivalent:
		\begin{enumerate}
			\item  A family $(\mathbf{x}_{sru})_{0\leq r<s<u}$ in  $\mathcal L_V$ is a two-way flow.
			\item There exist $H,G\in L$  such that $(1-r)(1-s)H+rs G+\mathrm{Id}$ is an invertible
			element of  $L$
			for all $0\le r<s$ and
			\begin{equation}\label{sru} \mathbf{x}_{sru}=
				\left( \tfrac{u-s}{u-r} B_{ru}^{-1} B_{su}, \tfrac{s-r}{u-r} B_{ur}^{-1} B_{sr}\right),\quad 0\le r<s<u,
			\end{equation}
			with
			$B_{\alpha\beta}=(1-\alpha)(1-\beta)H+\alpha\beta H_{\beta}G H_{\beta}^{-1}+\mathrm{Id}$,
			where $H_{\beta}=(1-\beta)H+\mathrm{Id}$.
		\end{enumerate}
	\end{proposition}
	\begin{proof}
		(i) $\Rightarrow$ (ii): Following \eqref{hr} we get
		$$
		\bh_r=r(1-r)(G,H)+(1-r)(\mathrm{Id},0)+r(0,\mathrm{Id})=((1-r)(rG+\mathrm{Id}),rH_r).
		$$
		Thus
		$$
		\bh_r^{-\ltimes}=\left(-\tfrac{1-r}{r}\,(rG+\mathrm{Id})H_r^{-1},\,\tfrac{1}{r}H_r^{-1}\right).
		$$
		
		Consequently, \eqref{definicjahs} yields
		\begin{align*}
			\bw_{ru}=\bh_u^{-\ltimes}\ltimes
			\bh_r&=\left(-\tfrac{1-u}{u}\,(uG+\mathrm{Id})H_u^{-1},\,\tfrac{1}{u}H_u^{-1}\right)\ltimes((1-r)(rG+\mathrm{Id}),rH_r)\\
			&=\left(\tfrac{u-r}{u}\,b_{ru}H_u^{-1},\,\tfrac{r}{u}H_rH_u^{-1}\right),
		\end{align*}
		for $b_{ru}=b_{ur}=(1-r)(1-u)H+ruG+\mathrm{Id}$.
		
		Thus,
		$$
		\bw_{ru}^{-\rtimes}=\left(\tfrac{u}{u-r}H_ub_{ru}^{-1},\,-\tfrac{r}{u-r}H_rb_{ru}^{-1}\right)
		$$
		and
		\begin{align}\nonumber
			\bw_{ru}^{-\rtimes}\rtimes\bw_{su}&=\left(\tfrac{u}{u-r}H_ub_{ru}^{-1},\,-\tfrac{r}{u-r}H_rb_{ru}^{-1}\right)\rtimes
			\left(\tfrac{u-s}{u}\,b_{su}H_u^{-1},\,\tfrac{s}{u}H_sH_u^{-1}\right)\\
			&=\left(\tfrac{u-s}{u-r}H_ub_{ru}^{-1}b_{su}H_u^{-1},\,\tfrac{s}{u}H_sH_u^{-1}-\tfrac{r(u-s)}{u(u-r)}\,H_rb_{ru}^{-1}b_{su}H_u^{-1}\right).\label{seco}
		\end{align}
		
		Since
		$$
		s(u-r)b_{ru}H_s=u(s-r)b_{rs}H_u+r(u-s) b_{su} H_r
		$$
		the second coordinate of \eqref{seco} can be written as $\tfrac{s-r}{u-r}H_rb_{ru}^{-1}b_{rs}H_r^{-1}$.\\
		Thus \eqref{sru} follows on noting that $H_{\beta}b_{\alpha\beta}H_{\beta}^{-1}=B_{\alpha\beta}$ is invertible in $L(V)$.
		\\
		
		(ii) $\Rightarrow$ (i): Note that $(G,H)$, for which $b_{rs}$ is invertible for all $0\leq r <s$, defines a
		family $(\bw_{ru})_{0\leq r<u}$ of $\rtimes$-invertible elements. Furthermore $\bw_{ru}^{-\rtimes}$ has $\ltimes$-inverse
		for all $0< r <u$  and
		\begin{equation*}
			\begin{split}
				\tfrac{1}{u(u-r)}(\bw_{ru}^{-\rtimes})^{-\ltimes}&+\tfrac{1}{ru}{\bf e}_{\ltimes}-\tfrac{1}{r(u-r)}{\bf
					e}_{\rtimes}\\
				&=\tfrac{1}{u(u-r)}\big(\tfrac{u}{r}H_uH_r^{-1},-\tfrac{u-r}{r}b_{ru}H_r^{-1}\big)+\tfrac{1}{ru}{\bf
					e}_{\ltimes}-\tfrac{1}{r(u-r)}{\bf e}_{\rtimes}\\
				&=\big(\tfrac{1}{r(u-r)}(H_u-H_r)H_r^{-1},\tfrac{1}{ru}(H_r-b_{ru})H_r^{-1}\big).
			\end{split}
		\end{equation*}
		Since $b_{ru}=(1-u)H_r+u(\mathrm{Id}+rG)$, then
		\begin{equation*}
			\tfrac{1}{u(u-r)}(\bw_{ru}^{-\rtimes})^{-\ltimes}+\tfrac{1}{ru}{\bf e}_{\ltimes}-\tfrac{1}{r(u-r)}{\bf
				e}_{\rtimes}=\big(-\tfrac{1}{r}H H_r^{-1},\tfrac{1}{r}\mathrm{Id}-\tfrac{1}{r}H_r^{-1}-GH_r^{-1}\big)
		\end{equation*}
		does not depend on $u$. So from Proposition %
		\ref{P-suff} we get that $(G,H)$ is a flow generator and  Theorem \ref{T-flow-exists} yields that
		$(\mathbf{x}_{sru})_{0\leq r<s<u}$ is a two-way flow.
	\end{proof}

	The following generalization of  \eqref{triv} %
	illustrates the fact that the definition of the two-way flow on $[0,\infty)$  is not symmetric with respect to $\oL$ and
	$\oR$.
	
	\begin{proposition} \label{2wayf} In a DNA, consider $\bh=\beta \mathbf{e}_\oR+\gamma\mathbf{e}_\oL$. Then $\bh$ is a flow
		generator if and only if
		$\gamma\in[-1,0]$ and $\beta+\gamma\geq 0$.
		
		Equivalently, $\left\{\mathbf{x}_{sru}\in \mbox{\rm span}\{\mathbf{e}_\oR,\mathbf{e}_\oL\}:\, 0\leq r<s<t\right\} $ is a
		two-way flow if and only if  there exist $\rho\in[0,1]$, $\alpha\geq 0$
		such that
		\begin{equation} \label{xsru-lala}
			\mathbf{x}_{sru}= \tfrac{(u-s) b(s,u)}{(u-r) b(r,u)}\mathbf{e}_\oR+\tfrac{(s-r) b(r,s)}{(u-r) b(r,u)}\mathbf{e}_\oL.
		\end{equation}
		with   $b(s,u)=\alpha su+\rho(s+u)+1-\rho$.
		
	\end{proposition}

	\begin{proof} We note that %
		$\mbox{\rm span}\{\mathbf{e}_\oR,\mathbf{e}_\oL\}$
		with induced multiplications forms
		a sub-DNA that is isomorphic to
		$\mathcal L_{\R}$. By  Proposition \ref{P-p}, $\bh$ is a flow generator if and only if $(1-r)(1-s)\gamma+r s \beta +1\ne 0$
		for all $0\leq r<s$.  Taking $r=0$ we see that we must have $(1-s)\gamma+1>0$ for all $s>0$, i.e. $\gamma\in[-1,0]$.
		Considering now large $r$, we see that we must have $\beta+\gamma\geq 0$. Conversely, if  $\gamma\in[-1,0]$ and
		$\beta+\gamma\geq 0$ then   $(1-r)(1-s)\gamma+r s \beta +1=b(r,s)>0$ for all $s>r \geq  0$  as    $\rho=-\gamma\in[0,1]$
		and $\alpha=\beta+\gamma\geq 0$.
		Formula \eqref{sru} gives \eqref{xsru-lala}.
	\end{proof}

	\subsection{ Two way flows in the  DNA
		for quadratic harnesses} \label{QH}
	The following DNA is implicit in the study of conditional variances  in \cite{bib_BrycMatysiakWesolowski}.
	Consider a linear space $(V,+):=(\mathbb{R}^6\times \mathbb{R}^2,+)$,
	for which elements will be written in the form
	\begin{equation}
		\label{X6}
		{\mathbf{x} \choose\mathbf{u}}=\binom{
			x_1,\ldots,x_6}{u_1,u_2}.
	\end{equation}
	
	With   $\mathbf{x}=(x_1,\ldots, x_6)$,
	$\mathbf{y}~=~(y_1,\ldots, y_6)~\in~\mathbb{R}^6$ and $\mathbf{u}=(u_1,u_2)$, $\mathbf{v}=(v_1,v_2) \in \mathbb{R}^2$, we define multiplication operations $\ltimes$ and $\rtimes$
	by formulas:
	\small
	\begin{multline}
		\label{QHrtimes}	{\mathbf{x} \choose\mathbf{u}}\rtimes 	{\mathbf{y} \choose\mathbf{v}}
		:=	\\ \binom{
			x_1y_1, x_2y_1+u_1y_2, x_3y_1+u_2y_2+y_3, x_4y_1+u_1y_4, x_5y_1+u_2y_4+y_5,x_6y_1+y_6 }{
			u_1v_1,u_2v_1+v_2},
	\end{multline}
	\begin{multline}
		\label{QHltimes}	{\mathbf{x} \choose\mathbf{u}}\ltimes 	{\mathbf{y} \choose\mathbf{v}}:=\\	\binom{
			y_1+u_1y_2+x_1y_3, u_2y_2+x_2y_3, x_3y_3, x_4y_3+y_4+u_1y_5, x_5y_3+u_2y_5, x_6y_3+y_6 }{v_1+u_1v_2, u_2v_2}.
	\end{multline}
	\normalsize
	It is easy to check that $Q:=(V,+,\ltimes,\rtimes)$ is a DNA with   neutral elements
	\begin{equation}\label{eQH}\mathbf{e}_\ltimes=\binom{
			0, 0, 1, 0, 0, 0}{
			0, 1} \quad \textnormal{and} \quad \mathbf{e}_\rtimes=\binom{
			1, 0, 0, 0, 0, 0 }{
			1, 0}.
	\end{equation}
	The $\ltimes$-inverse of $\mathbf{x} \choose\mathbf{u} $ exists if and only if $x_3\neq 0$ and $u_2\neq 0$, and
	$\rtimes$-inverse
	exists if and only if  $x_1\neq 0$ and $u_1\neq 0$. These inverses are respectively equal %
	\begin{eqnarray}
		\label{QHinvL}
		{\mathbf{x} \choose\mathbf{u}}^{-\ltimes}&= &	\left( \begin{array}{c}
			\frac{x_2u_1-x_1u_2}{x_3u_2}, -\frac{x_2}{x_3u_2}, \frac{1}{x_3}, \frac{x_5u_1-x_4u_2}{x_3u_2},
			-\frac{x_5}{x_3u_2},-\frac{x_6}{x_3}  \\
			-\frac{u_1}{u_2},\frac{1}{u_2}
		\end{array} \right),\\ 	
		\label{QHinvR}
		{\mathbf{x} \choose\mathbf{u}}^{-\rtimes}&= &	\left( \begin{array}{c}
			\frac{1}{x_1}, -\frac{x_2}{x_1u_1}, \frac{x_2u_2-x_3u_1}{x_1u_1}, -\frac{x_4}{x_1u_1},
			\frac{x_4u_2-x_5u_1}{x_1u_1},-\frac{x_6}{x_1}  \\
			\frac{1}{u_1},-\frac{u_2}{u_1}
		\end{array} \right).
	\end{eqnarray}

	We remark that it is difficult to determine general conditions  for
	\begin{equation}\label{Hv.6}
		\bh=\binom{
			h_1,h_2,h_3,h_4,h_5,h_6}{
			g_1,g_2}\in Q
	\end{equation}
	to be a flow generator.
	But  the lower coordinates
	$(u_1,u_2)$ in \eqref{X6} %
	under multiplications $\ltimes$ and $\rtimes$ from Section \ref{DNAofAF}    behave like the elements of
	$\mathcal L_{\R}$, so by Proposition \ref{2wayf} a necessary condition   for $\bH$ to be a flow generator is the
	requirement that $g_2\in[-1,0]$ and $g_1+g_2\geq 0$.
	
	The following example shows that  these conditions are not sufficient and  the converse of Theorem \ref{Tw o  h} does not
	hold without additional assumptions.
	
	\begin{example}	\label{Kontrprzyklad}
		Consider $  \bh:=\binom{
			1,1,0,0,0,0}{
			1,0}\in V.$
		One can check that \eqref{xsru_def} gives a well defined  family $\{\mathbf{x}_{sru}:\;0\leq r<s<u\}$ with both inverses,
		but that the second flow equation in  (\ref{RL_mult}) fails, e.g., for   $r=1$, $s=2$, $t=3$, $u=4$.
	\end{example}

	We now apply Theorem \ref{Tw o  h} and Theorem \ref{T-flow-exists}   to   a family   of a special form.
	
	\begin{theorem}\label{Twierdzenieopodlozuprobabilistycznym}
		A family
		$$\bigg\{\mathbf{x}_{sru}=\left( \begin{array}{c}
			x_{sru}^{(1)}, \ldots, 	x_{sru}^{(6)},\\
			\frac{u-s}{u-r}, \frac{s-r}{u-r}
		\end{array} \right): 0\leq r<s<u\bigg\}$$
		is a two-way flow if and only if there exist $\alpha\ge 0$,  $\rho\in[0,1]$, $\beta\geq -2\sqrt{\alpha(1-\rho)}$, and
		$h_4,h_5,h_6\in\mathbb{R}$ such that
        	\begin{multline}\label{x123}
        x^{(1)}_{sru}=\tfrac{(u-s)c(s,u)}{(u-r)c(r,u)}, \quad   x^{(2)}_{sru}=\tfrac{ (s-r) (u-s) (2\rho-\beta)}{(u-r)c(r,u)},
        \\
        x^{(3)}_{sru}=\tfrac{(s-r)c(r,s)}{(u-r)c(r,u)},  \quad  x^{(4)}_{sru}=\tfrac{(u-s)(s-r)(h_4 u-h_5(1-u))}{(u-r)c(r,u)},\\
        x^{(5)}_{sru}=\tfrac{(u-s)(s-r)(h_5(1-r)-h_4 r)}{(u-r)c(r,u)}, \quad  x^{(6)}_{sru}=\tfrac{(u-s)(s-r)h_6}{c(r,u)},
        \end{multline}
        %
        %
		where $ c(s,u)=\alpha su+(\beta-\rho)s+\rho u+1-\rho$, $0\le s<u$.
	\end{theorem}
	Equivalently,
	\begin{equation}
		\label{HHHH}
		\bh=
		\binom{ \alpha +\beta -\rho ,2 \rho -\beta ,-\rho
			,h_4,h_5,h_6}{0,0}
	\end{equation}
	is a flow generator if and only if $\alpha\ge 0$,  $\rho\in[0,1]$, $\beta\geq -2\sqrt{\alpha(1-\rho)}$.
	
	\begin{proof}%
		,,$\Rightarrow$''     With $\bH$ as in \eqref{Hv.6}
		we get
		$$ \bH_u=(1-u)\binom{1+ uh_1 ,uh_2 ,u
			\left(h_3+\tfrac{1}{1-u}\right),uh_4 ,uh_5 ,uh_6
		}{1+ug_1 ,u
			\left(g_2+\tfrac{1}{1-u}\right)}.
		$$
		We use now Theorem \ref{Tw o  h}. Since the lower coordinates   of the flow is actually the flow in the DNA $\mathcal
		L_{\R}$ of Section \ref{DNAofAF}, we conclude from \eqref{xsru-lala}
		with %
		the appropriate parameters that the lower coordinates of ${\bf x}_{10u}$ are
		$$\tfrac{(u-1) (g_1  u+1)}{u (1-g_2  (u-1))}\quad\mbox{ and }\quad \tfrac{1}{u (1-g_2  (u-1))}$$
		for all $u>1$. However by our assumptions they are
		$(u-1)/u$ and $1/u$,
		respectively. Therefore, $g_1=g_2=0$.
		
		Denoting now $\alpha:=h_1+h_2+h_3$, $\beta:=-h_2-2h_3$, $\rho:=-h_3$ and again using \eqref{xsru_def} and
		\eqref{definicjahs} to calculate ${\bf x}_{sru}$ we see that the lower coordinates are as they should be  and the upper
		coordinates are as given in \eqref{x123}.
		
		By $\rtimes$-invertibility of $\bH_u$ we have $1-(1-u)\rho\neq 0$ for any $u>0$ and thus $\rho\in[0,1]$.
		
		Since ${\bf x}_{sru}$  is  $\ltimes$-invertible it follows that $c(s,u)\neq 0$. Consequently, by its definition it follows
		that  $c(s,u)>0$ for small $0 \leq s<u$ %
		and by continuity the inequality holds for any $0 \leq s<u$. Taking large values of $0 \leq s<u$ we conclude that
		$\alpha\ge 0$. Note that $c(s,u)$ can be rewritten as $c(s,u)=\alpha s^2+\beta s+1-\rho+(u-s)(\alpha s+\rho)$. Since the
		last term  is nonnegative and  vanishes at $u=s$, considering separately the case $\alpha=\rho=0$ and $\alpha+\rho>0$ we
		conclude that  $c(s,u)>0$ for all $0 \leq s<u$ if
		$\alpha s^2+\beta s+1-\rho\geq 0$ for all $ s \geq 0$, i.e. when  $\beta\geq -2\sqrt{\alpha(1-\rho)}$.
		
		,,$\Leftarrow$''  Under assumed constraints on the  parameters, $c(s,u)>0$ for all $0\leq s<u$, so formulas in \eqref{x123} show that $\mathbf{x}_{sru}$ are $\ltimes$-invertible and $\rtimes$-invertible for all $0\leq r<s<u$.
		A lengthy calculation  shows that flow equations \eqref{RL_mult} are satisfied for all $0\leq r<s<u$.
		To verify condition \eqref{S-C} we confirm by calculation that
		\small
		\begin{align*}
			&\tfrac{{\bf x}_{sru}-\mathbf{e}_\ltimes}{u-s}+\tfrac{{\bf x}_{sru}-\mathbf{e}_\rtimes}{s-r}=\tfrac{1}{c(r,u)}\times\\
			&
			\times\binom{
				\beta -\rho +\alpha  u , 2 \rho -\beta  , -\alpha r-\rho  ,
				h_4 u+h_5 (u-1) , h_5 (1-r) - h_4 r, h_6 (u-r)}{  0 , 0 }
		\end{align*}
		\normalsize
		does not depend on $s\in(r,u)$.
	\end{proof}
	\section{ Applications to stochastic processes}\label{Sect:AP}
	In this section we  give some probabilistic applications of the
	previous results.
	We encounter a technical difficulty that the arising flows are on $(0,\infty)$ instead of $[0,\infty)$, which we
	overcome by   replacing the time variables $r,s,t$ with $r+p,s+p,t+p$ for $p>0$. This method introduces additional
	solutions
	of flow equations that do not extend to $[0,\infty)$.
	\subsection{Harnesses}
	Let $(X_t)_{t\geq0}$ be a separable integrable stochastic process with   complete   $\sigma$-fields
	$\mathcal{F}_{r,u}:=\sigma\{X_t: t\in[0,r]\cup [u,\infty)\}$ and  linear regressions
	\begin{equation}
		\label{LR-gen}
		\mathbb E(X_s|\mathcal{F}_{r,u})=a_{sru}X_r+b_{sru}X_u,
	\end{equation}
	where coefficients $a_{sru}$ and $b_{sru}$ are deterministic and depend only on   $0\leq r<s<u$.
	Following Mansuy and Yor \cite{bib_Yor} we say that  $(X_t)_{t\geq0}$ is a {\em  harness}, if \eqref{LR-gen} holds with
	\begin{equation}
		\label{atsu-btsu}
		a_{sru}=\tfrac{u-s}{u-r}  \quad  \mbox{ and }  \quad b_{sru}=\tfrac{s-r}{u-r},
	\end{equation}
	see \eqref{LR0}.
	Here is a probabilistic application of Proposition \ref{corI} to harnesses with quadratic regressions,
	compare   \cite{CKT-2012}. %
	\begin{proposition}\label{C.2.3}
		Suppose that a square integrable process   $(X_t)_{t\geq 0}$ is a harness with moments $\mathbb E\,X_t=0$,
		$\mathbb E\,X_sX_t=s\wedge t  $, and that $X_s^2,\,X_s, 1$ are linearly independent for any $s>0$.
		Let $\mathcal{F}_t=\sigma\{X_s: s\in[0,t]\}$, and assume that for any $0<s<t$ there exist non-random $a_{ts}\ne 0$,
		$b_{ts}$ and $c_{ts}$ such that
		\begin{equation}\label{E-sq}
			\mathbb E(X_t^2|\mathcal{F}_s)=a_{ts}X_s^2+b_{ts}X_s+c_{ts}.
		\end{equation}
		\begin{enumerate}
			\item If $s\mapsto a_{ts}$ is  bounded  on $(0,t] $   for some  $t>0$, then
			there exist $b\in\mathbb R$ and $a\ge 0$ such that
			\begin{equation}\label{ab1}
				a_{ts}=\tfrac{1+ta}{1+sa},\quad b_{ts}=\tfrac{(t-s)b}{1+sa},\quad c_{ts}=\tfrac{t-s}{1+sa}.
			\end{equation}
			\item If $s\mapsto a_{ts}$ is  unbounded  on $(0,t] $   for some $t>0$, then
			there exists $b\in\R$ such that
			\begin{equation}\label{ab2}
				a_{ts}=\tfrac{t}{s}, \quad b_{ts}=\tfrac{(t-s)b}{s}, \quad c_{ts}=0.
			\end{equation}
			
		\end{enumerate}
		
	\end{proposition}
	
	\begin{proof}
		We first observe that passing to the limit in \eqref{LR-gen} as $u\to\infty$,  see   \cite[proof of
		(4)]{bib_BrycWesolowski}, we get
		\begin{equation}
			\label{LR1}
			\mathbb E(X_t|\mathcal{F}_s)= X_s.
		\end{equation} %
		We will relate the problem to a flow in the near-algebra
		$\mathcal A_{V}$ from Section \ref{Ex_of_DNA} with $V=\RR^2$  and multiplication given in \eqref{QHproduct}.  Within this framework we will apply
		Proposition \ref{corI}.   There is however a technical difficulty caused by
		lack of linear independence for $s=0$, as
		$X_0=0$, so  we consider
		${\bf x}_{st}=(a_{ts},\,b_{ts},\,c_{ts})$ for $0<s<t<u$, i.e., we exclude $s=0$.
		Since $a_{ts}\ne 0$ by assumption, ${\bf x}_{st}$ is $\boxdot$-invertible.
		
		Note that \eqref{LR1} gives
		\begin{multline*}
			a_{us}X_s^2+b_{us}X_s+c_{us}=\mathbb E(X_u^2|\mathcal{F}_s)=\mathbb E(\mathbb E(X_u^2|\mathcal{F}_t)|\mathcal{F}_s)
			\\=a_{ut}\mathbb E(X_t^2|\mathcal{F}_s)+b_{ut}\mathbb E(X_t|\mathcal{F}_s)+c_{ut}
			=a_{ut}(a_{ts}X_s^2+b_{ts}X_s+c_{ts})+b_{ut}X_s+c_{ut}.
		\end{multline*}
		Comparing coefficients at linearly independent functions $X_s^2$, $X_s$ and 1, we see that flow equation
		\eqref{proflow} holds  for all $0<s<t<u$.
		
		Using \eqref{LR1} again, we have
		\begin{equation}
			\label{EXTU1}
			\mathbb E(X_tX_u|\mathcal{F}_s)=\mathbb E(X_t^2|\mathcal{F}_s)=a_{ts}X_s^2+b_{ts}X_s+c_{ts}.
		\end{equation}
		On the other hand,  harness property \eqref{LR-gen} with \eqref{atsu-btsu} implies
		\begin{multline}  \label{EXTU2}
			\mathbb E(X_tX_u|\mathcal{F}_s)=\mathbb E(\mathbb
			E(X_t|\mathcal{F}_{s,u})\,X_u|\mathcal{F}_s)=\tfrac{u-t}{u-s}X_s\mathbb E(X_u|\mathcal{F}_s)+\tfrac{t-s}{u-s}\mathbb
			E(X_u^2|\mathcal{F}_s)\\=\tfrac{u-t}{u-s}X_s^2
			+\tfrac{t-s}{u-s}\left(a_{us}X_s^2+b_{us}X_s+c_{us}\right).
		\end{multline}
		Comparing \eqref{EXTU1} and \eqref{EXTU2}  we obtain \eqref{prolin} for all $0<s<u$.
		We now fix $p>0$ and apply
		Proposition \ref{corI} %
		to the flow $\{\widetilde{\mathbf{x}}_{st}:0\leq s<t\}$ defined by
		$\widetilde{\mathbf{x}}_{st}= \mathbf{x}_{s+p,t+p}$
		Formula \eqref{22}
		shows that there exist constants $a_p\geq  0, b_p,c_p\in\RR$ such that
		for all $t>s>p$ we have
		\begin{equation}\label{arbrcr}
			a_{ts}=\frac{1+a_p (t-p)}{1+a_p(s-p)},\quad  b_{ts}=\frac{(t-s)b_p}{1+a_p(s-p)},\quad  c_{ts}=\frac{(t-s)c_p}{1+a_p(s-p)}.
		\end{equation}
		 Recall that the original family ${\bf x}_{st}=(a_{ts},\,b_{ts},\,c_{ts})$, $0\leq s<t$, does not depend on $p$.
		We now observe  that
		\begin{equation}
			\label{P1} \tfrac{t-s a_{ts}}{t-s} = \tfrac{p a_p-1}{a_p (p-s)-1},
		\end{equation}
		 where the left hand side does not depend on $p\in (0,s)$.
		Since we can take $s>0$ arbitrarily large this  means that we have two cases:
		\begin{itemize}
			\item Case A: $pa_p\ne 1$ for all $p>0$.
			\item Case B: $pa_p=1$ for all $p>0$.
			
		\end{itemize}
		In Case A,
		we use \eqref{P1} to write  $(t-s)/(t-s a_{ts})$ as $1+s a_p/(1-pa_p)$.
		Since this expression does not depend on $p$ when  $s>p$ is arbitrary,
		there is a constant $a$ such that $a_p/(1-p a_p)=a$ for all  $p>0$.
		We get  $a_p=a/(1+ap)$ for all small enough  $p>0$. As $a_p\geq 0$, this implies that  $a\geq 0$ and hence $a_p=a/(1+ap)$
		for all $p>0$.
		Dividing the last two equations in \eqref{arbrcr} by $t-s$ and applying \eqref{P1} we  see that  $(1+a p)  b_p=b$ and  $(1+a p)
		c_p=c$  do not depend on $p$.
		Inserting these expressions into \eqref{arbrcr} for all $0<s<t$ we get
		\begin{equation}\label{abc} a_{ts}=\tfrac{1+a t}{1+as},\quad  b_{ts}=\tfrac{(t-s)b}{1+as},\quad
			c_{ts}=\tfrac{(t-s)c}{1+as}.
		\end{equation}
		From \eqref{abc} it is clear that Case A does not arise if coefficients $a_{st}$ are unbounded for some $t$.
		
		Since $\mathbb E\,X_t^2=t$, taking expected value of \eqref{E-sq}, at some $t>s>0$ we get
		$$ t= \tfrac{1+a t}{1+as} s+ \tfrac{(t-s)c}{1+as}, $$
		which gives $c=1$. As $X_0=0$, the coefficients of the quadratic form on the right hand side of  \eqref{EXTU1}
		are given by formula \eqref{ab1}    for all $0\leq s<t$.

		In Case B we have $a_p=1/p$ for all $p>0$. From \eqref{arbrcr} we get  $a_{ts}=t/s$ for all $t>s$.
		It is therefore clear that Case B does not arise if the coefficient $a_{st}$ is bounded near $s=0$ for some $t$.
		
		Taking the expected value of both sides of \eqref{E-sq}, we get $c_{ts}=0$ for all $t>s>0$. Formula \eqref{arbrcr} gives
		$b_{ts}=\frac{(t-s)p b_p}{s}$ so  $p b_p=b$ does not depend on $p>0$ and \eqref{ab2} follows.
	\end{proof}

	\subsection{More general linear regressions}	\label{S:LR}

	In this section we are interested in a separable integrable stochastic process
	such that $\mathbb{E}X_t=0$ with the property that  for all $0<r<s<u$   we have
	
	\begin{equation}\label{Z-harness}
		\mathbb{E}\left(\left.\tfrac{X_s-X_r}{s-r}-\tfrac{X_u-X_s}{u-s}\right|\mathcal{F}_{r,u}\right)=A_{ru}X_r+B_{ru}X_u,
	\end{equation}
	where $A_{ru}$ and $B_{ru}$ are deterministic functions  of $r$ and $u$ but not of  $s$.  Note that  \eqref{Z-harness}
	defines
	a harness, see \eqref{LR-gen} and \eqref{atsu-btsu}, when
	$A_{ru}=B_{ru}\equiv 0$. %

	\begin{theorem}
		\label{tw-G-harnesik}
		Let $(X_t)_{t\geq0}$ be an integrable centered stochastic process such that \eqref{Z-harness} holds.   Suppose that
		$X_r$ and $X_s$  are  linearly independent as functions on the probability space $\Omega$ for  	$0< r<s$.
		\begin{enumerate}
			\item If $r\mapsto A_{ru}$ is  bounded  on $(0,u] $ for some  $u>0$, then
			there exists $\alpha\geq 0$ and  $\rho\in[0,1]$, such that for all $0\leq r<s<u$ equation \eqref{LR-gen} holds with
			coefficients $a_{sru}$ and $b_{sru}$  given by
			\begin{equation}\label{xsru-R2}
				a_{sru}=\tfrac{(u-s) b(s,u)}{(u-r)
					b(r,u)}, \qquad\qquad
				b_{sru}=    \tfrac{(s-r) b(r,s)}{(u-r) b(r,u)},
			\end{equation}
			with $b(s,u)$ defined in   Proposition \ref{2wayf}.
			\item  If $r\mapsto A_{ru}$ is  unbounded  on $(0,u] $ for some  $u>0$, then
			equation \eqref{LR-gen} holds for all $0< r<s<u$  with coefficients $a_{sru}$ and $b_{sru}$  given by
			\begin{equation}\label{xsru-R3}
				a_{sru}= \tfrac{s (u-s)}{r (u-r)},
				\qquad\qquad
				b_{sru}=\tfrac{s (s-r )}{u (u-r )}.
			\end{equation}
			
		\end{enumerate}
		
	\end{theorem}

	\begin{remark}
		Recalculating $A_{ru}$ and $B_{ru}$  from either \eqref{xsru-R2} or \eqref{xsru-R3}, in view of \eqref{a2Ab2B},  we get
		\begin{itemize}
			\item[(i)] $A_{ru}=\tfrac{\alpha u+\rho}{b(r,u)}$, $B_{ru}=-\tfrac{\alpha r+\rho}{b(r,u)}$;
			\item[(ii)] $A_{ru}=\tfrac{1}{r}$, $B_{ru}=-\tfrac{1}{u}$.
		\end{itemize}
	\end{remark}

	\begin{proof}[Proof of Theorem \ref{tw-G-harnesik}]
		Clearly  \eqref{Z-harness} implies that \eqref{LR-gen} holds with
		\begin{equation}
			\label{a2Ab2B}
			a_{sru}=\tfrac{(u-s)(s-r)}{u-r}\left(\tfrac{1}{s-r}+A_{ru}\right), \quad
			b_{sru}=\tfrac{(u-s)(s-r)}{u-r}\left(\tfrac{1}{u-s}+B_{ru}\right)
		\end{equation}  for all $0< r<s<u$.
		From the tower property of conditional expectation  we have for $0< r<s<t<u$:
		\begin{multline}	\label{rown-pom1}
			\mathbb{E}(X_t|\mathcal{F}_{r,u})
			=\mathbb{E}\left(\mathbb{E}(X_t|\mathcal{F}_{s,u})\mid\mathcal{F}_{r,u}\right)
			\\=\mathbb{E}\big(a_{tsu}X_s+b_{tsu}X_u\big|\mathcal{F}_{r,u}\big)=a_{tsu}a_{sru}X_r+(a_{tsu}b_{sru}+b_{tsu})X_u,
		\end{multline}
		and
		\begin{multline}	\label{rown-pom2}
			\mathbb{E}(X_s|\mathcal{F}_{r,u})
			=\mathbb{E}\left(\mathbb{E}(X_s|\mathcal{F}_{r,t})\mid\mathcal{F}_{r,u}\right)
			\\=\mathbb{E}\big(a_{srt}X_r+b_{srt}X_t\big|\mathcal{F}_{r,u}\big)=(a_{srt}+a_{tru}b_{srt})X_r+b_{srt}b_{tru}X_u.
		\end{multline}
		 Consider  DNA $\mathcal L_\RR$ %
		from Section \ref{DNAofAF}, which is isomorphic to
		the span of  $\mathbf{e}_\oR=(1,0),\mathbf{e}_\oL=(0,1)$, and recall  Proposition \ref{2wayf}.
		Define a family $(\mathbf{x}_{sru})_{0< r<s<u}$ by $\mathbf{x}_{sru}=(a_{sru},b_{sru})$.   By
		linear
		independence of $X_r$ and $X_u$, from \eqref{rown-pom1} and \eqref{rown-pom2} we obtain
		that this family satisfies flow equations \eqref{RL_mult} for all
		$0< r<s<t<u$.
		
		 We note that
		$$\mathbb{E}\left(\left.\tfrac{X_s-X_r}{s-r}-\tfrac{X_u-X_s}{u-s}\right|\mathcal{F}_{r,u}\right)=(\tfrac{a_{sru}-1}{s-r}+\tfrac{a_{sru}}{u-s})X_r+(\tfrac{b_{sru}}{s-r}+\tfrac{b_{sru}-1}{u-s})X_u.$$
		Condition \eqref{Z-harness} implies that the left hand side does not depend on $s\in(r,u)$. Therefore
		$$\tfrac{(a_{sru},b_{sru})-(1,0)}{s-r}+\tfrac{(a_{sru},b_{sru})-(0,1)}{u-s}$$
		does not depend on $s\in(r,u)$. Consequently, $(\mathbf{x}_{sru})_{0< r<s<u}$ satisfies the structure condition \eqref{S-C}.
		
		To ensure $\oR$-invertibility and $\oL$-invertibility, we need to verify that $a_{sru}\ne 0$ and  $b_{sru}\ne 0$.
		Suppose that $a_{sru}=0$ for some  $0< r<s<u$.  Then  from \eqref{rown-pom1} by linear independence we have $a_{tru}=0$
		holds for all
		$t\in(s,u)$.
		Hence from \eqref{Z-harness} we obtain
		$$A_{ru}X_r+B_{ru}X_u=\mathbb{E}\left(\tfrac{X_t-X_r}{t-r}-\tfrac{X_u-X_t}{u-t}\mid\mathcal{F}_{r,u}\right)=\tfrac{u-r}{(t-r)(u-t)}b_{tru}X_u-\tfrac{X_r}{t-r}-\tfrac{X_u}{u-t},
	    $$
		and using again the linear independence
		$A_{ru}=-1/(t-r)$
		holds for  $s<t<u$, which
		is impossible as $A_{ru}$ does not depend on $t$. %
		Therefore $a_{sru}\ne0$ for every $0< r<s<u$.
		Analogously from (\ref{rown-pom2}) we can show that $b_{sru}\ne0$ for all $0<  r<s<u$.
		
		This means that  $(\mathbf{x}_{sru})_{0< r<s<u}$ is a two-way flow on $(0,\infty)$. 	
		We now proceed as in the proof of Proposition \ref{C.2.3}. For arbitrary $p>0$, define $r'=r+p,s'=s+p,u'=u+p$, and
		consider $(\widetilde{\mathbf{x}}_{sru})_{0\leq r<s<u}$ given by $\widetilde{\mathbf{x}}_{sru}=(a_{s'r'u'},b_{s'r'u'})$. 	
		Proposition \ref{2wayf} shows
		that there exist $\alpha_p\geq 0$ and $\rho_p\in[0,1]$ such that with
		$$
		b_p(r,u)=
		\alpha_p r u + (\rho_p - p \alpha_p)(r+u)+1-\rho_p(1+2p)+p^2\alpha_p
		$$
		we have
		\begin{equation}\label{asru-bsru*}
			a_{sru}=\tfrac{(u-s) b_p(s,u)}{(u-r)
				b_p(r,u)}, \qquad
			b_{sru}=    \tfrac{(s-r) b_p(r,s)}{(u-r) b_p(r,u)},\qquad p\leq r<s<u.
		\end{equation}

		Thus
		$$A_{ru}=\tfrac{\alpha_p u+\rho_p-p\alpha_p}{b_p(r,u)}\quad \mbox{and}\quad
		B_{ru}=-\tfrac{\alpha_pr+\rho_p-p\alpha_p}{b_p(r,u)}.$$
		Recall that the original family ${\bf x}_{sru}=(a_{sru},\,b_{sru})$, $0\leq s<t$, does not depend on $p\in(0,r)$ and therefore in view of \eqref{a2Ab2B} we conclude that $A_{ru}$ and $B_{ru}$ also do not depend on $p\in(0,r)$. Furthermore we have
		
		\begin{equation}
			\label{Sand1}
			\frac{ A_{ru}+B_{ru}}{u-r}=\frac{\alpha_p}{b_p(r,u)}
		\end{equation}
		
		\begin{equation}
			\label{Sand2}
			\frac{uB_{ru}+r A_{ru}}{u-r}=\frac{p\alpha_p-\rho_p}{b_p(r,u)}
		\end{equation}
		 and both left hand sides above do not depend on $p\in(0,r)$.
		Since $r>0$ can be arbitrarily large,   in view of \eqref{Sand1} it follows that
		either $\alpha_p=0$ for all $p>0$ or $\alpha_p\ne 0$ for all $p>0$. We consider these two cases separately.
		\begin{case} $\alpha_p=0$ for all $p>0$.
		\end{case}
		In this case, the fact that \eqref{Sand2} does not depend on $p\in(0,r)$
		and $r$ can be taken arbitrarily large, this implies either $\rho_p=0$ for all $p>0$  and \eqref{xsru-R2}
		holds
		with $\alpha=0,\rho=0$,
		or after factoring out $\rho_p$ from  the denominator on the right hand side of \eqref{Sand2},  the   term
		$\frac{1}{\rho _p}-2 p=C\geq 1$ does not depend on $p>0$.  In the latter  case,
		\eqref{xsru-R2} holds  with $\alpha=0, \rho=1/C\in(0,1]$.
		
		\begin{case} $\alpha_p> 0$ for all $p>0$.
		\end{case}
		In this case, dividing \eqref{Sand2} by \eqref{Sand1} we see that
		$\rho_p/\alpha_p-p=C_1$ does not depend on $p>0$. Putting this into \eqref{Sand1} with $\alpha_p$ factored out from the
		denominator, we see  that the expression
		$$\tfrac{1}{\alpha_p}-(2 C_1 +1)p-p^2 =C_2$$
		in the denominator cannot depend on
		$p>0$.
		Hence
		\begin{equation}\label{alpha_p-rho_p}
			\alpha_p=\tfrac{1}{2 C_1 p+p+p^2+C_2}, \quad \rho_p=\tfrac{C_1+p}{2 C_1 p+p+p^2+C_2}.
		\end{equation}

		Note that since $\alpha_p>0$ and $\rho_p\in[0,1]$, see Proposition \ref{2wayf},
		taking the limit as $p\searrow 0$ in both expressions \eqref{alpha_p-rho_p}   we see that $0\leq C_1\leq C_2$.
		
		Substituting \eqref{alpha_p-rho_p} into \eqref{asru-bsru*} we get the answer which does not depend on $p>0$, with
		$$\tfrac{b_p(s,u)}{b_p(r,u)}=\tfrac{C_1 (s+u-1)+C_2+s u}{C_1 (r+u-1)+C_2+r u},\quad \tfrac{b_p(r,s)}{b_p(r,u)}=\tfrac{C_1
			(r+s-1)+C_2+r s}{C_1 (r+u-1)+C_2+r u}.$$
		If $C_1=C_2=0$, we get \eqref{xsru-R3} with unbounded $A_{ru}$. If $C_2>0$ then we get \eqref{xsru-R2}  with
		$$
		\alpha=1/C_2>0 ,\quad \rho= C_1/C_2\in[0,1].
		$$
		Combining the  cases, we see that either $A_{ru}$ is unbounded as $r\to  0$ and then \eqref{xsru-R3} holds, or $A_{ru}$ is
		bounded as $r\to  0$  and then \eqref{xsru-R2} holds with $\alpha\geq 0$ and $\rho\in[0,1]$.
	\end{proof}

	\subsection{ Quadratic harnesses}\label{S:App QH}
	Following \cite{bib_BrycMatysiakWesolowski}, we say that a separable square integrable stochastic process $(X_t)_{t\geq0}$
	is a {\em standard  quadratic harness},  if it is a harness,  $\mathbb{E}X_t=0$ and $\mathbb{E}X_sX_t=s\wedge t$ for
	$s,t\geq0$, and
	\begin{equation}
		\label{rown_kwadrat}
		\mathbb{E}(X_s^2|\mathcal{F}_{r,u})=A_{sru}X_r^2+B_{sru}X_rX_u+C_{sru}X_u^2+D_{sru}X_r+E_{sru}X_u+F_{sru},
	\end{equation}
	where deterministic coefficients $A_{sru},\ldots,F_{sru}$ depend only on $0\leq r<s<u$,
	see \eqref{QH0}.
	
	The following result extends \cite[Theorem 2.2]{bib_BrycMatysiakWesolowski}, who considered only case $\chi=1$.
	\begin{theorem} \label{Tw-prob-szer+}
		Let $(X_t)_{t\geq0}$ be a standard  quadratic harness. Suppose that  	$1$, $X_r$, $X_u$, $X_r X_u$,  $X_r^2$, $X_u^2$
		are linearly
		independent for all $0< r<u$   as functions on $\Omega$. Then there exist constants $\chi\in\{0,1\}$, $\sigma\geq 0$,
		$\tau\geq 0$, $\gamma\leq \chi+2\sqrt{\sigma\tau}$ and $\eta,\theta\in\RR$ such that
		$\gamma,\sigma,\tau,\chi$ are not all zero and
		\begin{subequations}
			\begin{equation}\label{CV*}
				\var(X_s|\mathcal{F}_{r,u})=  \tfrac{(s-r) (u-s)}{c_{\tau,\sigma,\gamma,\chi}(r,u)}
				Q_{\theta,\eta,\tau,\sigma,\gamma-\chi,\chi}\left(\tfrac{X_u-X_r}{u-r},\tfrac{uX_r-rX_u}{u-r}\right), \quad 0< r<u
			\end{equation}
			with
			\begin{equation}\label{cc*}
				c_{\tau,\sigma,\gamma,\chi}(r,u)=\tau +\sigma r u-\gamma  r+\chi u
			\end{equation}
			and
			\begin{equation}
				\label{QQ*}
				Q_{\theta,\eta,\tau,\sigma,\rho,\chi}(x,y)= \tau  x^2+\sigma  y^2+\rho  x y+\theta  x+  \eta  y +\chi.
			\end{equation}
		\end{subequations}

	\end{theorem}
	\begin{remark}\label{R-chi}
		The values of the two-valued parameter $\chi$ are determined by the last term in \eqref{rown_kwadrat} as follows:
		\begin{enumerate}
			\item If $F_{sru}\ne0$ for some $0<r<s<u$ then $\chi=1$ and hence $F_{sru}\ne0$ for all $0<r<s<u$.
			\item If $F_{sru}=0$ for some $0<r<s<u$   then $\chi=0$ and hence $F_{sru}=0$ for all $0<r<s<u$.
		\end{enumerate}
		By homogeneity of  \eqref{CV*} if $\chi=0$,   the remaining parameters   are defined only up to an arbitrary multiplicative
		factor.
		Therefore, if $\tau>0$,  we can take
		$\tau=1$ as in  Lemma \ref{Lemma-p*}; or if   $\sigma>0$, we can take $\sigma=1$; or if    $\gamma< 0$  we can take	
		$\gamma=-1$. (These three cases arise  in the proof below.)
	\end{remark}
	\begin{remark}\label{R-deg}
		If $\chi=\tau=0$ and $\sigma>0$, or  if $\chi=\tau=\sigma=0$ and $\gamma<0$, then the conditional variance and most of
		the coefficients in \eqref{rown_kwadrat} are unbounded near $r=0$.
	\end{remark}

	\subsection*{Proof of Theorem \ref{Tw-prob-szer+}}
	We apply here Theorem \ref{Twierdzenieopodlozuprobabilistycznym}, however as in the proofs of Proposition \ref{C.2.3}
	and Theorem \ref{tw-G-harnesik} we encounter technical difficulties as the arising two-way flows  are on $(0,\infty)$
	instead of $[0,\infty)$ and, as indicated in Remark \ref{R-deg}, may not extend to $[0,\infty)$. We therefore begin
	with the following
	direct application of Theorem \ref{Twierdzenieopodlozuprobabilistycznym}.
	\begin{lemma}\label{Lemma-p*} For every $p>0$ there exist constants $\eta_p,\theta_p,\gamma_p\in\RR$,  $\chi_p\in\{0,1\}$,
		$\sigma_p\geq 0$, $\tau_p\geq 0$, where $\tau_p=1$ when $\chi_p=0$ and $\gamma_p\leq \chi_p+2\sqrt{\sigma_p\tau_p}$,  such
		that $\forall p<r<s<u$
		\begin{equation}\label{CVp*}
			\var(X_s|\mathcal F_{r,u})=  \tfrac{(s-r) (u-s)}{c_{\tau_p^*,\sigma_p,\gamma_p^*,\chi_p^*}(r,u)}
			Q_{\theta_p^*,\eta_p,\tau_p^*,\sigma_p,\gamma_p^*-\chi_p^*,\chi_p^*}
			\left(\tfrac{X_u-X_r}{u-r},\tfrac{uX_r-rX_u}{u-r}\right),
		\end{equation}
		where
		\begin{equation}
			\label{params*}
			\tau_p^*=\tau_p+p(\gamma_p-\chi_p)+p^2\sigma_p, \; \chi_p^*=\chi_p-p\sigma_p, \;  \gamma_p^*=\gamma_p+p\sigma_p,  \;
			\theta_p^*=\theta_p+p\eta_p,
		\end{equation}
		while $c$ and $Q$ are defined in \eqref{cc*} and \eqref{QQ*}, respectively.
	\end{lemma}
	\begin{proof}
		Recall notation  \eqref{atsu-btsu}, \eqref{rown_kwadrat} and DNA from Section \ref{QH}. For $0<s<t<u$    let
		$$\mathbf{x}_{tsu} :=\binom{
			A_{tsu}, B_{tsu}, C_{tsu}, D_{tsu}, E_{tsu}, F_{tsu}}{
			a_{tsu}, b_{tsu}}\in \mathbb{R}^6 \times \mathbb{R}^2.$$
		We proceed to verify that $\{\mathbf{x}_{sru}\}$ is a two-way flow on open half-line $(0,\infty)$ with respect to
		multiplications \eqref{QHrtimes} and \eqref{QHltimes}.
		
		Because
		$\mathbb{E}(X_sX_t|\mathcal{F}_{r,u})=\mathbb{E}\left(X_s\mathbb{E}(X_t|\mathcal{F}_{s,u})|\mathcal{F}_{r,u}\right)=\mathbb{E}\left(\mathbb{E}(X_s|\mathcal{F}_{r,t})X_t|\mathcal{F}_{r,u}\right)$, by linear independence of $1$, $X_s$, $X_t$,
		$X_sX_t$, $X_s^2$, $X_t^2$ for $0<r<s<t<u$, we have (see e.g. proof of Claim 3.1 in \cite{bib_BrycMatysiakWesolowski}):
		\begin{equation}
			\begin{split}	\label{ukl-rown}
				a_{tsu}A_{sru}=b_{srt}A_{tru}+a_{srt}a_{tru},\quad
				a_{tsu}B_{sru}&=b_{srt}B_{tru},\\
				b_{tsu}b_{sru}+a_{tsu}C_{sru}=b_{srt}C_{tru},\quad
				a_{tsu}D_{sru}&=b_{srt}D_{tru}, \\
				a_{tsu}E_{sru}=b_{srt}E_{tru}, \quad
				a_{tsu}F_{sru}&=b_{srt}F_{tru}.
			\end{split}
		\end{equation}	
		
		Since
		$
		a_{srt}a_{tru}=\tfrac{(t-s)(u-t)}{(t-r)(u-r)}, \quad b_{tsu}b_{sru}=\tfrac{(t-s)(s-r)}{(u-s)(u-r)} $     and
		$$\tfrac{u-t}{u-r}=\tfrac{(t-s)(u-t)}{(t-r)(u-r)}+\tfrac{(s-r)(u-t)}{(t-r)(u-r)},\quad
		\tfrac{(u-t)(s-r)}{(u-s)(u-r)}+\tfrac{(t-s)(s-r)}{(u-s)(u-r)}=\tfrac{s-r}{u-r},
		$$
		we see that
		$$\tfrac{u-t}{u-s}\mathbf{x}_{sru}+\tfrac{(t-s)(s-r)}{(u-s)(u-r)}\mathbf{e}_\ltimes=\tfrac{(t-s)(u-t)}{(t-r)(u-r)}\mathbf{e}_\rtimes+\tfrac{s-r}{t-r}\mathbf{x}_{tru}$$
		holds for all $0\leq s<t<u$, with $\mathbf{e}_\ltimes$ and $\mathbf{e}_\rtimes$   defined by \eqref{eQH}.
		Hence
		
		$$\tfrac{\mathbf{x}_{sru}-\mathbf{e}_\ltimes}{u-s}+\tfrac{\mathbf{x}_{sru}-\mathbf{e}_\rtimes}{s-r}=\tfrac{\mathbf{x}_{tru}-\mathbf{e}_\ltimes}{u-t}+\tfrac{\mathbf{x}_{tru}-\mathbf{e}_\rtimes}{t-r}$$
		holds and the structural condition (\ref{S-C}) is satisfied.
		
		Moreover, because  %
		\begin{equation*}
			\mathbb{E}(X_t^2|\mathcal{F}_{r,u})=\mathbb{E}\left(\mathbb{E}(X_t^2|\mathcal{F}_{s,u})|\mathcal{F}_{r,u}\right),
			\label{rownanko1}
		\end{equation*}
		  we have (see e.g. proof of Claim 3.2 in \cite{bib_BrycMatysiakWesolowski}):
		\begin{equation*}
			\begin{split}
				A_{tru}=A_{tsu}A_{sru},\quad
				B_{tru}&=A_{t,s,u}B_{sru}+B_{tsu}a_{sru},\\
				C_{tru}=A_{tsu}C_{sru}+B_{tsu}b_{sru}+C_{tsu},\quad
				D_{tru}&=A_{tsu}D_{sru}+D_{tsu}a_{sru},\\
				E_{tru}=A_{tsu}E_{sru}+D_{tsu}b_{sru}+E_{tsu},\quad
				F_{tru}&=A_{tsu}F_{sru}+F_{tsu}.
			\end{split}
		\end{equation*}
		The above conclusion holds for $r>0$ by linear independence of $1$, $X_s$, $X_t$, $X_sX_t$, $X_s^2$,
		$X_t^2$, leaving constants $ A_{s0u}$
		$B_{s0u}$ and $D_{s0u}$ arbitrary or undefined as $X_0=0$.
		Hence the first flow equation in (\ref{RL_mult}) is satisfied on $(0,\infty)$.

		If $A_{s_0r_0u_0}=0$ for some  $0<r_0<s_0<u_0$, then from the first equation above we get that
		$A_{t_0r_0u_0}=0$ for all $t_0$ such that $s_0<t_0<u_0$. Then from the first equation in (\ref{ukl-rown}) we obtain
		$a_{s_0r_0t_0}a_{t_0r_0u_0}=0$, which contradicts \eqref{atsu-btsu}.  Hence $A_{sru}\neq0$ for all $0< r<s<u$, and
		$\mathbf{x}_{sru}$ is $\rtimes$-invertible.
		
		Because  %
		\begin{equation*}
			\mathbb{E}(X_s^2|\mathcal{F}_{r,u})=\mathbb{E}\left(\mathbb{E}(X_s^2|\mathcal{F}_{r,t})|\mathcal{F}_{r,u}\right),
			\label{rownanko2}
		\end{equation*}
		then analogously we can obtain that the second flow equation  (\ref{RL_mult}) is satisfied on $(0,\infty)$, where
		operation $\ltimes$
		is defined in \eqref{QHltimes}. Furthermore, by a similar argument  $C_{sru}\neq0$ holds for every $0< r<s<u$,
		so $\mathbf{x}_{sru}$ is $\ltimes$-invertible. 	
		
		This means that  $(\mathbf{x}_{sru})_{0< r<s<u}$ is a two-way flow on $(0,\infty)$. 	
		We now proceed as in the proofs of Proposition \ref{C.2.3} and Theorem \ref{tw-G-harnesik}. For arbitrary $p>0$, define
		$r'=r+p,s'=s+p,u'=u+p$, and
		consider $(\widetilde{\mathbf{x}}_{sru})_{0\leq r<s<u}$ given by $\widetilde{\mathbf{x}}_{sru}=\mathbf{x}_{s'r'u'}$.
		Then
		$(\widetilde{\mathbf{x}}_{tsu})_{0\leq s <t<u}$     is a two-way flow on $[0,\infty)$ with (unique) flow generator %
		\eqref{HHHH} that was determined
		in the proof of Theorem \ref{Twierdzenieopodlozuprobabilistycznym}. Of course, $\bh$ now depends on parameter $p>0$.
		We now re-parameterize $\bh$ using
		$\eta_p,\theta_p,\sigma_p,\tau_p,\gamma_p,\phi_p\in\RR$ such that %
		\begin{equation*}
		\begin{split}
		\alpha+\beta-\rho &= \tfrac{\sigma_p -\gamma_p}{\chi_p+\tau_p}, \; 2\rho-\beta=\tfrac{\gamma_p+\chi_p}{\chi_p+\tau_p},\;
		\rho=\tfrac{\chi_p}{\chi_p+\tau_p},\\
		\;h_4&=\tfrac{\eta_p-\theta_p}{\chi_p+\tau_p}, \;h_5=\tfrac{\theta_p}{\chi_p+\tau_p},
		\;h_6=\tfrac{\phi_p}{\chi_p+\tau_p}
		.
		\end{split}
		\end{equation*}
		Here $\chi_p=1_{\rho=0}$  takes only two values $0,1$; without loss of generality
		we may assume $\chi_p+\tau_p>0$,   taking $\tau_p=1$ when $\chi_p=0$.   In this parametrization $\bh$ takes the following
		form
		$$\bh=\frac{1}{\tau_p+\chi_p}\binom{
			{\sigma_p -\gamma_p }, {\gamma_p +\chi_p} ,- \chi_p , \eta_p -\theta_p ,\theta_p , \phi_p}{
			0,0}.$$
		Condition $\rho\in[0,1]$ is then equivalent to $\tau_p\geq 0$, condition $\alpha\geq 0$ is
		equivalent to $\sigma_p\geq 0$ and condition $\beta\geq -2\sqrt{\alpha(1-\rho)}$ is equivalent to $\gamma_p\leq
		\chi_p+2\sqrt{\sigma_p\tau_p}$.

		Formulas in \eqref{x123} give
		\begin{eqnarray}
			\label{A*}   A_{sru}&=&\tfrac{(u-s) c_{\tau_p^*,\sigma_p,\gamma_p^*,\chi_p^*}(s,u)}{(u-r)
				c_{\tau_p^*,\sigma_p,\gamma_p^*,\chi_p^*}(r,u)},\\
			\label{B*}    	B_{sru}&=&\tfrac{(s-r) (u-s)(\gamma_p^* +\chi_p^*) }{(u-r)c_{\tau_p^*,\sigma_p,\gamma_p^*,\chi_p^*}(r,u)},
			\\
			\label{C*}   	C_{sru}&=& \tfrac{(s-r)
				c_{\tau_p^*,\sigma_p,\gamma_p^*,\chi_p^*}(r,s)}{(u-r)c_{\tau_p^*,\sigma_p,\gamma_p^*,\chi_p^*}(r,u)},\\
			\label{D*}     	D_{sru}&=&\tfrac{(s-r) (u-s) (\eta_p  u-\theta_p^*)}{(u-r)
				c_{\tau_p^*,\sigma_p,\gamma_p^*,\chi_p^*}(r,u)},\\
			\label{E*}     E_{sru}&=&\tfrac{(s-r) (u-s) (\theta_p^* -\eta_p  r)}{(u-r)c_{\tau_p^*,\sigma_p,\gamma_p^*,\chi_p^*}(r,u)},
			\\
			\label{F*}  	F_{sru}&=&\tfrac{(s-r) (u-s)\phi_p  }{c_{\tau_p^*,\sigma_p,\gamma_p^*,\chi_p^*}(r,u)},
		\end{eqnarray}
		with $\tau_p^*$, $\gamma_p^*$, $\chi_p^*$ and $\theta_p^*$ as given in \eqref{params*}.
		Moreover, from the assumptions that $\mathbb{E}X_t=0$ and $\mathbb{E}X_sX_t=s\wedge t$ we get
		\begin{equation*}	\label{QH-sol*}
			s=\mathbb{E}X_s^2=\mathbb{E}\,\mathbb{E}(X_s^2|\mathcal{F}_{r,u})=A_{sru}r+B_{sru}r+C_{sru}u+F_{sru}.
		\end{equation*}
		Using   \eqref{A*} --\eqref{C*} and \eqref{F*}, after some algebra, we get
		\begin{equation}
			\label{chi2sigma-phi*}
			\phi_p=\chi_p-p\sigma_p=\chi_p^*.
		\end{equation}
		Since $\var(X_s|\mathcal F_{r,u})=\mathbb E(X_s^2|\mathcal F_{r,u})+\left[\mathbb E(X_s|\mathcal F_{r,u})\right]^2$
		elementary but tedious calculations give  \eqref{CVp*}.
	\end{proof}

	By linear independence of $1$, $X_r$, $X_u$, $X_r^2$, $X_rX_u$, $X_u^2$ for $0<r<u$ it follows from \eqref{CVp*}  that the
	vector
	\begin{equation}
		\label{TAMS*}
		v^*_p:=\tfrac{(\sigma_p,\,\tau_p^*,\, \gamma_p^*,\,\eta_p,\,\theta_p^*,\,\chi _p^*)}
		{c_{\tau_p^*,\sigma_p,\gamma_p^*,\chi_p^*}(r,u)}
	\end{equation}
	does not depend on $p>0$,  in the sense that $v^*_{p_1}=v^*_{p_2}$ for
	all positive  $p_1,p_2$ when  $r<u$ are arbitrary in $[p_1\vee p_2,\infty)$.
	
	Using this fact   we conclude that if either $\chi_p^*=0$, or $\sigma_p=0$, or $\gamma_p^*=0$
	for some $p>0$, then, respectively,   $\chi_p^*=0$, or $\sigma_p=0$, or $\gamma_p^*=0$ for all $p>0$.
	We therefore have the following four  cases which we need to consider separately:
	\begin{itemize}
		\item Case A: $\chi_p^*\ne 0$ for all $p>0$.
		\item Case B:
		$\chi_p^*=0$   and $\sigma_p>0$ for all $p>0$.
		\item Case C:
		$\chi_p^*=0$,  $\sigma_p=0$ and $\gamma_p^*=0$ for all $p>0$.
		\item Case D:
		$\chi_p^*=0$, $\sigma_p=0$ and $\gamma_p^*\ne 0$  for all $p>0$.
	\end{itemize}
	
	\setcounter{case}{1}
	\setcounter{case}{0}
	\begin{case}
		$\chi_p^*\ne 0$ for all $p>0$.
	\end{case}
	\setcounter{claim}{0}
	\begin{claim} There exist constants $\sigma\geq 0$,  $\tau\geq 0$ and $\gamma\leq 1+2\sqrt{\sigma\tau}$,
		$\eta,\theta\in\RR$ such that for all $p>0$
		\begin{equation}\label{sigma*}
			(\sigma_p,\,\tau_p^*,\,\gamma_p^*,\,\eta_p,\,\theta_p^*,\,\chi_p^*)=(1+p\sigma)(\sigma,\tau,\eta,\theta,\gamma,1).
		\end{equation}
	\end{claim}
	\begin{proof}
		In this case, after dividing  first five
		components
		of $v^*_p$, see \eqref{TAMS*}, by its last %
		component
		we see that there exist constants $\sigma, \tau, \gamma, \eta,\theta \in\RR$  (independent of $p$) such that
		$$
		\left(\sigma_p,\, \tau_p^*,\,\gamma_p^*,\,\eta_p,\,\theta_p^*\right)=\chi_p^*\,(\sigma, \tau, \gamma, \eta,\theta).
		$$
		Plugging the relation $\sigma_p=\sigma \chi_p^*$ into \eqref{chi2sigma-phi*} we get $\chi_p^*(1+p\sigma)=\chi_p$.
		Consequently, since $\chi_p\in\{0,1\}$ and $\chi_p^*\ne0$ it follows that $\chi_p:=\chi=1$.  Thus $\chi_p^*>0$ at least for
		small $p>0$. Combining this with $\sigma_p=\sigma \chi_p^*$ and with the fact that  $\sigma_p\ge 0$ we get
		$\sigma\geq 0$, so $1+p\sigma>0$ for any $p>0$. Hence  we conclude that \eqref{sigma*} hold for all $p>0$.
		
		Note that $\lim_{p\to 0^+}\tau_p^*=\lim_{p\to 0^+}\tau_p=\tau$. Since $\tau_p\geq0$ for all $p\ge 0$ we get $\tau\geq 0$.
		Similarly, $\lim_{p\to 0^+}\gamma_p^*=\lim_{p\to 0^+}\gamma_p=\gamma$ and $\lim_{p\to 0^+}\sigma_p=\sigma$.
		Since $\gamma_p\leq 1+2\sqrt{\sigma_p\tau_p}$ for all $p>0$ we get $\gamma\leq 1+2\sqrt{\sigma\tau}$.
	\end{proof}
	\setcounter{claim}{0}
	
	Substituting  \eqref{sigma*} into \eqref{CVp*}  we immediately get
	\eqref{CV*}--\eqref{QQ*} with $\chi=1$.

	\begin{case}
		$\chi^*_p=0$   and $\sigma_p>0$ for all $p>0$.
	\end{case}
	\begin{claim}There exist constants $\eta,\theta\in\RR$, $\tau\geq 0$, $\gamma\leq 2\sqrt{\tau}$   such that for all $p>0$
		\begin{equation}
			\label{chiS}(\sigma_p,\,\tau_p^*,\,\gamma_p^*,\,\eta_p,\,\theta_p^*,\chi_p^*)=\tfrac{1}{p}(\sigma,\,\tau,\,\gamma,\,\eta,\,\theta,\,0).
		\end{equation}
	\end{claim}
	\begin{proof}
		From \eqref{chi2sigma-phi*} we get
		$\chi_p=p \sigma_p>0$ for all $p>0$. As $\chi_p\in\{0,1\}$, this means that   $\chi_p=1$ and $\sigma_p=1/p$.
		Dividing components %
		2-5 of $v_p^*$ by the first component, see \eqref{TAMS*},  we see that there exist constants $\tau,
		\gamma,\eta,\theta\in\RR$ such that
		$$p(\tau_p^*,\,\gamma_p^*,\,\eta_p,\,\theta_p^*)=(\tau, \gamma,\eta,\theta).$$
		Thus we    get   \eqref{chiS}. Since $\tau_p\geq0$  we see that $\tau\geq 0$.
		Furthermore, $\gamma_p\le 2\sqrt{\tau_p\sigma_p}$ yields $\gamma-p\le 2\sqrt{\tau+p^2-p\gamma}$. %
		Taking $p\to 0$ we get $\gamma\le 2\sqrt{\tau}$.
	\end{proof}
	
	Substituting  \eqref{chiS} into \eqref{CVp*}  we immediately get
	\eqref{CV*}--\eqref{QQ*} with $\chi=\sigma=1$.

	\begin{case}
		$\chi_p^*=0$, $\sigma_p=0$ and   $\gamma_p^*= 0$ for all $p>0$.
	\end{case}
	\setcounter{claim}{0}
	\begin{claim} There are constants $\eta,\theta\in\RR$ such that for all $p>0$
		\begin{equation}
			\label{chiS22} (\sigma_p,\,\tau_p^*,\,\gamma_p^*,\,\eta_p,\,\theta_p^*,\,\chi_p^*)=(0,1,0,\eta ,\theta ,0).
		\end{equation}
	\end{claim}
	\begin{proof}
		Note $\gamma_p^*=\gamma_p=0$. Moreover,  $\tau_p=1$ yields $\tau_p^*=1$. Thus \eqref{chiS22} follows  from the fact that
		in \eqref{TAMS*} vector $v_p^*=(0,1,0,\eta_p,\theta_p^*,0)$  does not depend on $p>0$.
	\end{proof}
	Substituting  \eqref{chiS22} into \eqref{CVp*}   we get
	\eqref{CV*}-\eqref{QQ*} with $\chi=0, \sigma=0,\gamma=0,\tau=1$.

	\begin{case}
		$\chi_p^*=0$, $\sigma_p=0$ and $\gamma_p^*\neq 0$ for all $p>0$.
	\end{case}
	\setcounter{claim}{0}
	\begin{claim} There exist $\tau\geq  0$, and $\eta,\theta\in\RR$ such that
		for all $p>0$
		\begin{equation}
			\label{chiS33}  (\sigma_p,\,\tau_p^*,\,\gamma_p^*,\,\eta_p,\,\theta_p^*,\,\chi_p^*)=\tfrac{1}{p+\tau}
			\left(0, {\tau } ,-1, {\eta } , {\theta } , 0\right).
		\end{equation}
	\end{claim}
	\begin{proof}  By Lemma \ref{Lemma-p*} we have $\tau_p=1$, $\gamma_p^*=\gamma_p$ and $\chi_p=0$.
		Dividing  the remaining  components of $v_p^*$ by its third component, see \eqref{TAMS*}, we  get
		$$(\tau_p^*,\eta_p,\theta_p^*)=-\gamma_p^*(\tau,\eta,\theta)$$
		for some  constants $\tau,\eta,\theta$. Since  $\tau_p^*=1+p\gamma_p$, comparing the first components in the above equation
		we get $\gamma_p=-\tfrac{1}{\tau+p}$, at least for small $p>0$.  This identity together with the fact that $\gamma_p\le 0$
		for all $p>0$ yields $\tau\ge 0$. Hence the identity holds for all $p>0$, and \eqref{chiS33} follows.
	\end{proof}
	
	Substituting  \eqref{chiS33} into \eqref{CVp*}   we get
	\eqref{CV*}-\eqref{QQ*} with $\chi=0, \sigma=0,\gamma=-1$.

	\begin{remark}
		A conjecture in \cite{Bryc-Wesolowski-09} says that when $\chi=1$, the remaining parameters determine the distribution of
		the quadratic harness uniquely. The following example shows that uniqueness may fail when $\chi=0$.
		
		Let $(X_t)_{t\geq0}$ be a Markov process defined   in \cite[Section 4.1]{bib_BrycWesolowski} as the $q$-Brownian process
		with $q=-1$
		, i.e.,  \eqref{CV*} holds with parameters $\chi=1$, $\theta=\eta=\sigma=\tau=0$ and $\gamma=-1$.   It is known that
		$(X_t)_{t\geq0}$ is a standard harness and $|X_t|=\sqrt{t}$  for all  $t\geq0$.   Hence linear independence assumption
		fails,
		and we can write the quadratic form \eqref{QQ*} in many ways.
		
		Let $(Z_t)_{t\geq0}$ be a stochastic process  given by $Z_t:=Y \cdot X_t,$ where $Y$ is a random variable such that $Y$ and
		$(X_t)_{t\geq0}$ are independent and $\mathbb{E}Y^2=1$. Then $\mathbb{E}Z_t=0$, $\mathbb{E}Z_sZ_t=s\wedge t$. Let
		$(\mathcal{F}_{s,u})_{0\leq s<u}$ and $(\mathcal{G}_{s,u})_{0\leq s<u}$ be filtrations for $(X_t)_{t\geq0}$ and
		$(Z_t)_{t\geq0}$ respectively. Then
		\begin{equation*}
			\begin{split}
				\mathbb{E}\big(Z_t\big|\mathcal{G}_{s,u}\big)&=\mathbb{E}\big(Y\mathbb{E}\big(X_t\big|Y,\mathcal{F}_{s,u}\big)\big|\mathcal{G}_{s,u}\big)=
				\mathbb{E}\big(Y\mathbb{E}\big(X_t\big|\mathcal{F}_{s,u}\big)\big|\mathcal{G}_{s,u}\big)\\	
				&=\mathbb{E}\big(Y\big(\tfrac{u-t}{u-s}X_s+\tfrac{t-s}{u-s}X_u\big)\big|\mathcal{G}_{s,u}\big)=\tfrac{u-t}{u-s}Z_s+\tfrac{t-s}{u-s}Z_u,
			\end{split}
		\end{equation*}
		hence $(Z_t)_{t\geq0}$ is a harness. Because
		$X_t^2=t=\frac{u-t}{u-s}s+\frac{t-s}{u-s}u=\frac{u-t}{u-s}X_s^2+\frac{t-s}{u-s}X_u^2$ holds, then for all $0\leq s<t<u$ we
		have
		$$\mathbb{E}\big(Z_t^2\big|\mathcal{G}_{s,u}\big)=\mathbb{E}\big(t
		Y^2\big|\mathcal{G}_{s,u}\big)=\tfrac{u-t}{u-s}Z_s^2+\tfrac{t-s}{u-s}Z_u^2,$$
		so $(Z_t^2)_{t\geq0}$ is also a harness. From the above form, it follows that
		$(Z_t)_{t\geq0}$ is a quadratic harness with parameters  $\chi=\theta=\eta=\gamma=\sigma=0$ and $\tau=1$. The
		distribution of $Z_t$ is not unique as it depends on the choice of $Y$.
		
	\end{remark}

	\subsection*{Acknowledgements}
	The authors thank the anonymous referee for a  thorough report that improved the
	readability of the paper.
	This work was supported by a grant from the Simons Foundation/SFARI Award Number: 703475 (WB) and  grant
	2016/21/B/ST1/00005 of National Science Centre, Poland (JW).
	
	
	

\end{document}